\documentclass[12pt]{amsart}
\usepackage{epsfig,color,wrapfig}

\makeatother

\theoremstyle{definition}

\newtheorem{thm}{Theorem}[section]
\newtheorem{cor}[thm]{Corollary}

\newtheorem{lem}[thm]{Lemma}

\newtheorem{rmk}[thm]{Remark}
\newtheorem{prop}[thm]{Proposition}

\numberwithin{equation}{section}

\address{National Taiwan University, Department of Mathematics\\
No.1, Sec. 4, Roosevelt Rd., Da'an Dist., Taipei City, Taiwan, 10617}
\email{jechang@ntu.edu.tw}
\address{National Center for Theoretical Science, Mathematics Division\\
No.1, Sec. 4, Roosevelt Rd., Da'an Dist., Taipei City, Taiwan, 10617}
\email{yklue@ncts.ntu.edu.tw}

\title{Uniqueness of regular shrinkers with 2 closed regions}
\author{Jui-En Chang, Yang-Kai Lue}
\date{\today}

\usepackage{latexsym, amsmath,amssymb}
\usepackage{amsmath}
\usepackage{amsfonts}
\usepackage{amssymb}
\usepackage{amscd}
\usepackage{amsthm}
\usepackage{amsbsy}
\usepackage{graphicx}
\usepackage{bm}
\usepackage{epsfig,epstopdf,url,array}

\begin{document}
\maketitle
\begin{abstract}
Regular shrinkers describe blow-up limits of a finite-time
singularity of the motion by curvature of planar network of curves.
This follows from Huisken's monotonicity formula. In this paper, we
show that there is only one regular shrinker with 2 closed regions.
This regular shrinker is the Cisgeminate eye. Moreover, we find some
degenerate regular shrinkers with 2 closed regions.
\end{abstract}

\section{Introduction}

A regular network is an embedded network which satisfies the Herring
condition: all multi-points are of degree 3 and the angles between
curves are $\frac{2\pi}{3}$. The reader can refer to \cite{MNPS} for
detail. Given an initial regular network $\Gamma_0$, a network flow
is a family of networks with Herring condition and fixed boundary
points that satisfy that $(\frac{\partial X_i}{\partial
t})^\perp=\bar{k_i}$. Here $X_i$ is the position vector and
$\bar{k_i}$ is the curvature vector of the curve $\gamma_i$.
Recently, many researchers have studied this flow in
\cite{AGH,BN,INS,MMN,MNP,MNP2,MNPS,MNT,P}.

The short time existence of this flow of an initial regular $C^2$
network with a triple junction is proved by L. Bronsard and F.
Reitich in \cite{BR}. Recently, the short time existence of this
flow of an initial regular $C^2$ network with multiple junctions is
proved by C. Mantegazza, M. Novaga, and A. Pluda in \cite{MNP2}.
Using a parabolic rescaling procedure at the singular time and
Huisken's monotonicity formula \cite{Hu}, there is a subsequence
which converges to a possibly degenerate regular network. This limit
network shrinks self-similarly to the origin and it may be an open
network. An open regular network is called a regular shrinker if it
satisfies
\begin{equation}
\bar{k}+x^\perp=0
\end{equation}
at any point, where $\bar{k}$ is the curvature vector. A regular
shrinker will move by homothety with respect to the origin under the
network flow. Such a network describes the behavior of the flow at
the singular time.

We are interested in the classification of regular shrinkers. If
there are no triple junctions, the network flow is the curve
shortening flow and the self-similarly shrinking solution of the
flow is described in the work of U. Abresch and J. Langer \cite{AL}.
They classify all immersed curves and show that the only embedded
self-similarly shrinking curves are a line or a circle. A regular
shrinker with exactly 1 triple junction must be a standard triod or
a Brakke spoon, where the latter one is first described in the work
of K. Brakke \cite{B}. The Brakke spoon is shown to be the blow-up
limit for all spoon-shaped network in the work of Pluda \cite{P}.
The classification of regular shrinker with 1 closed region is done
by X. Chen and J. -S. Guo \cite{CG}. From the work of Mantegazza,
Novaga, and Pluda \cite{MNP}, for an evolving network with at most
two triple junctions, the multiplicity-one conjecture holds. P.
Baldi, E. Haus, and Mantegazza \cite{BHM,BHM2} exclude the
$\Theta$-shaped network. Together with the work by Chen and Guo
\cite{CG}, all regular shrinkers with 2 triple junctions are
completely characterized. There are only 2 such networks: the lens
and the fish. This classification is used to study the general
behavior of networks with 2 triple junctions in the work of
Mantegazza, Novaga, Pluda \cite{MNP}. The lens is shown to be the
rescaling limit of any flow starting from a symmetric lens-shaped
network in \cite{AGH} and the work of G. Bellettini and Novaga
\cite{BN}. The appendix of \cite{MNPS} contains a collection of all
known regular shrinkers and some possible numerical results.

Apart from the cases described above, the classification of regular
shrinkers remains open. In this paper, we complete the
classification of all regular shrinkers with 2 closed regions. We
establish the following result.
\begin{thm}
The only regular shrinker with 2 closed regions is the Cisgeminate
eye.
\end{thm}
\begin{figure}[h]
  \centering
    \includegraphics[width=5cm]{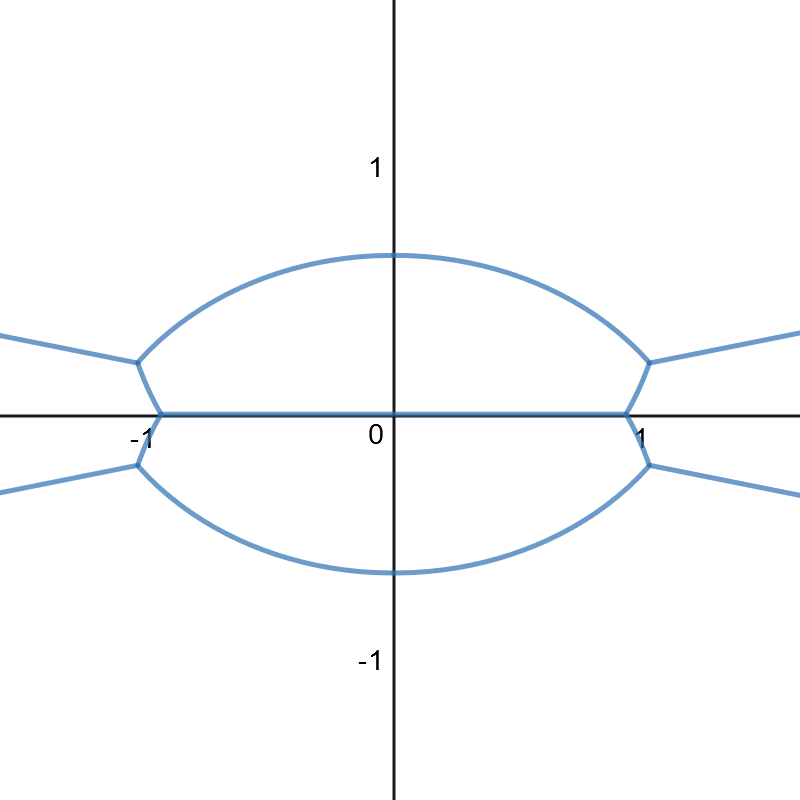}
  \caption{Cisgeminate eye proposed in \cite{MNPS}}
\end{figure}

The paper is organized as follows. For any regular shrinker, it must
be Abresch-Langer curves which intersect at triple junctions with
angle $\frac{2\pi}{3}$. In section 2, we introduce the phase space
to describe the behavior of Abresch-Langer curves. We also define
some terminology which will be used throughout this article. In
section 3, we focus on the possible topology of such networks and
show that the topology must be a $\Theta$-shaped network with rays
attached. Among the 2 closed regions, at least one of them does not
contain the origin. In section 4, using the estimation of change of
angle in \cite{BHM2}, we are able to show the region which does not
contain the origin must be a 4-cell. Therefore, the topology of the
network must be a 4-cell attached to either a 2,3,4 or 5 cell. In
section 5, we eliminate the possibility for the other cell to be
either a 5-cell or a 2-cell. In section 6, we deal with the
remaining case and establish the uniqueness of such a network. In
section 7, we relax the condition to allow the regular shrinker to
be degenerate and find some solutions for degenerate regular
shrinkers. Some of the solutions may have curves with multiplicity
greater than 1.

\section*{Acknowledgement}
The authors would like to thank Prof. Yng-Ing Lee for valuable
discussion.

\section{Phase plane of Abresch-Langer curves}
For a curve $\gamma(s)$ going around the origin in the
counterclockwise direction, let $R$ be the distance to the origin
and $\theta$ be the angle in polar coordinates. Let $\psi$ be the
signed angle from $\gamma$ to $\gamma_s$. We have
$0\leq\psi\leq\pi$. The following expression in terms of $R-\psi$ is
derived in the work of Chen and Guo \cite{CG}. For any curve, from
the definition of $\psi$, we hav e $\frac{dR}{ds}=\cos\psi$,
$\frac{d\theta}{ds}=\frac{1}{R}\sin\psi$. For $\sin\psi\neq0$,
dividing the equations yields $\frac{dR}{d\theta}=R\cot\psi$. Let
$\phi$ be the angle of the unit tangent vector. On a self-similarly
shrinking curve, we have $\frac{d\phi}{ds}=k=\langle
\gamma,N\rangle=R\sin\psi$. Note that $\phi=\theta+\psi$. Therefore,
$\frac{d\psi}{ds}=\frac{d\phi}{ds}-\frac{d\theta}{ds}=(R-\frac{1}{R})\sin\psi$.
Combining the equation involve $R$ and $\psi$, we have
$(R-\frac{1}{R})dR=\cot\psi d\psi$. Therefore, on a self-similarly
shrinking curve, we have
\begin{equation}\label{eq:conservation_law}
K(R)=c\sin\psi,
\end{equation}
for some $c\geq 1$, where $K(R)$ is given by
\begin{equation}
K(R)=\frac{\exp(\frac{R^2-1}{2})}{R}.
\end{equation}
We define $c$ to be the energy of the curve. For the special case
$\sin\psi=0$, $\theta$ is a constant and the solution is a line
through the origin. We define the energy for such curve to be
infinite.

From now on, we call a curve which satisfies $\bar{k}+x=0$ an
AL-curve. Define $R-\psi$ plane as the phase plane and we will
consider the trajectory $K(R)=c\sin\psi$ for some $c\geq1$. The
function $K(R)$ is strictly decreasing in $(0,1)$, strictly
increasing in $(1,\infty)$ and attains its absolute minimum 1 at
$R=1$. Therefore, $\psi$ attains the maximum
$\pi-\sin^{-1}(\frac{1}{c})$ and the minimum
$\sin^{-1}(\frac{1}{c})$ at $R=1$.

Using the phase plane, we want to compute the change of angle
$\theta$ when we move from one point to another point on the
trajectory. If we use $R$ as the variable, it can be expressed as
\begin{equation}\label{eq:int_R_for_theta}
\Delta\theta=\int_{R_1}^{R_2}\frac{d\theta}{dR}dR=\int_{R_1}^{R_2} \frac{\tan\psi}{R} dR=\int_{R_1}^{R_2}\frac{K(R)}{R\sqrt{c^2-K(R)^2}}dR.
\end{equation}
Note that if we fix $R_1$ and $R_2$, $\Delta\theta$ is monotonically
decreasing with respect to $c$.

There are expressions of $\Delta\theta$ in terms of other variables.
$\Delta\theta$ and $\Delta\phi$ are related by
$\Delta\theta=\Delta\phi-\Delta\psi$, where $\Delta\psi$ can be
determined by the starting and the ending point on the phase plane.
Let $\eta=1+2\log c$. Taking log in both side of the equation
(\ref{eq:conservation_law}), we obtain another expression of
conservation law with respect to $\eta$.
\begin{align}\label{eq:conservation_law1}
R^2-2\log k=\eta.
\end{align}
Consider the lower half of the trajectory where
$0<\psi<\frac{\pi}{2}$. Since
$\frac{dk}{ds}=\frac{d}{ds}(R\sin\psi)=R^2\cos\psi\sin\psi$, using
the conservation law (\ref{eq:conservation_law1}), it gives
\begin{equation}
\Delta\phi=\int_{k_1}^{k_2}\frac{d\phi}{dk}dk=
\int_{k_1}^{k_2}\frac{1}{R\cos\psi}dk=
\int_{k_1}^{k_2}\frac{1}{\sqrt{R^2-k^2}}dk=
\int_{k_1}^{k_2}\frac{1}{\sqrt{\eta-V(k)}}dk,
\end{equation}
where $V(k)=k^2-2\log k$ and the third equality comes from
$R^2\cos^2\psi=R^2-R^2\sin^2\psi=R^2-k^2$. The potential $V(k)$
attains its minimum at $k=1$. Also, for a fixed $\eta$, we define
$k_{\mathbf{min}}$ to be the unique $k<1$ which satisfies
$V(k)=\eta$. The variable $\eta$ can be regarded as the energy in
terms of $k$. Note that this equation is derived in \cite{C}.

The following are expressions in terms of $\psi$. Since the
trajectory is not symmetric with respect to the $R=1$ line, we need
to deal with $R<1$ case and $R>1$ case separately. Let $R=R^-(s)$
and $R=R^+(s)$ be the two inverses of $s=K(R)$. The domains of
$R^-$, $R^+$ are both $(1,\infty)$. The range of $R^-$ and $R^+$ are
$(0,1)$, $(1,\infty)$ respectively. The change of angle,
$\Delta\theta$, is given by
\begin{equation}
\begin{split}
\Delta\theta&=\int_{\psi_1}^{\psi_2}\frac{d\theta}{d\psi}d\psi=\int_{\psi_1}^{\psi_2}\frac{d\psi}{1-[R^-(c\sin\psi)]^2},\\
\Delta\theta&=\int_{\psi_1}^{\psi_2}\frac{d\theta}{d\psi}d\psi=\int_{\psi_1}^{\psi_2}\frac{d\psi}{[R^+(c\sin\psi)]^2-1},\\
\end{split}
\end{equation}
for $R<1$ and $R>1$, respectively.

\begin{lem}\label{lem:left_right_ineq}
For any $x>1$, we have $1-[R^-(x)]^2<[R^+(x)]^2-1$. Therefore, for
$\sin^{-1}(\frac{1}{c})\leq\psi_1<\psi_2\leq\pi-\sin^{-1}(\frac{1}{c})$,
\begin{equation}\label{ineq:LRineq}
\int_{\psi_1}^{\psi_2}\frac{d\psi}{1-[R^-(c\sin\psi)]^2}>\int_{\psi_1}^{\psi_2}\frac{d\psi}{[R^+(c\sin\psi)]^2-1}.
\end{equation}
\end{lem}
\begin{proof}
Let $V(R)=R^2-2\log R$ and $\eta=1+2\log(x)$, we have
$V(R^+)=V(R^-)=\eta$. Let $\alpha=1-R^-$.  We have $0<\alpha<1$ and
define
$f(\alpha)=V(1+\alpha)-V(1-\alpha)=4\alpha-2\log(1+\alpha)+2\log(1-\alpha)$.
Since $\frac{d}{d\alpha}f(\alpha)=\frac{-4\alpha^2}{1-\alpha^2}<0$
for $0<\alpha<1$, we have $f(\alpha)< f(0)=0$. Therefore,
$V(R^+)=V(R^-)=V(1-\alpha)=V(1+\alpha)-f(\alpha)>V(1+\alpha)$. This
means $R^+>1+\alpha$ and $1-R^-=\alpha<R^+-1$. We obtain
\begin{equation}\label{ineq:left_right_diff}
1-(R^-)^2=(1-R^-)(R^-+1)<(R^+-1)(R^++1)=(R^+)^2-1.
\end{equation}
The inequality (\ref{ineq:LRineq}) is an immediate consequence of
the inequality (\ref{ineq:left_right_diff}).
\end{proof}

Now, we consider the behavior of the network at a triple junction.
On a trajectory satisfying $K(R)=c\sin\psi$, the points where
$\psi=\frac{\pi}{3}$, $\frac{2\pi}{3}$ are important. Define $A(c)$,
$B(c)$, $C(c)$, $D(c)$ be the points on the trajectory with
coordinates $(R^+(\frac{\sqrt{3}}{2}c),\frac{\pi}{3})$,
$(R^+(\frac{\sqrt{3}}{2}c),\frac{2\pi}{3})$,
$(R^-(\frac{\sqrt{3}}{2}c),\frac{2\pi}{3})$,
$(R^-(\frac{\sqrt{3}}{2}c),\frac{\pi}{3})$ respectively. Also,
define $M(c)=(1,\pi-\sin^{-1}(\frac{1}{c}))$,
$N(c)=(1,\sin^{-1}(\frac{1}{c}))$ to be the points with extreme
$\psi$ value.
\begin{figure}[ht]
\includegraphics[height=4cm]{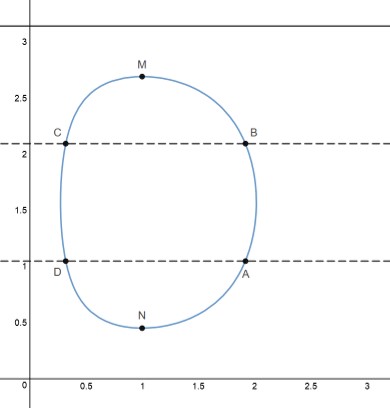}
\caption{The points A, B, C, D, M, N on the trajectory}
\end{figure}
\begin{lem}\label{same_energy}
If one of the AL-curves into a triple junction in a regular shrinker
is a ray or a line segment, the other 2 curves must have the same
energy. Therefore, if we move in the counterclockwise direction, the
corresponding point on the phase plane must switch from A to B or
from D to C on the trajectory at such a triple junction.
\end{lem}
\begin{proof}
At the triple junction, consider the value of $(R,\psi)$ for each of
the curves. Since the ray and the other 2 curves pass through this
triple junction, the $R$ value are the same. For a ray, we have
$\psi=0$ or $\pi$. The $\psi$ values are $\psi_1=\frac{\pi}{3}$ for
the incoming curve and $\psi_2=\frac{2\pi}{3}$ for the outgoing
curve. Therefore, the energy of the 2 curves are
\begin{equation}
    c_1=\frac{K(R)}{\sin\psi_1}=\frac{K(R)}{\sin\psi_2}=c_2.
\end{equation}
This argument also holds for the case that one AL-curve is a line
segment.
\end{proof}

\begin{rmk}
When $c<c_*=\frac{2}{\sqrt{3}}$, the trajectory does not intersect
the lines $\psi=\frac{\pi}{3}$, $\psi=\frac{2\pi}{3}$. In this case,
the points $A(c)$, $B(c)$, $C(c)$, $D(c)$ is undefined. However, we
can still define $M(c)$ and $N(c)$ where $R=1$ and $\psi$ attains
the extreme value.
\end{rmk}

From now on, for any 2 points $P$, $Q$ on the trajectory
$K(R)=c\sin\psi$, use the notation $\Delta\theta_{PQ}$ to express
the change of angle $\theta$ when we traverse the trajectory in the
counterclockwise direction from $P$ to $Q$ without achieving a
complete period. Define
\begin{equation}
\begin{split}
h_1(c)&=\Delta\theta_{CD}(c)=\int_\frac{\pi}{3}^\frac{2\pi}{3}\frac{d\psi}{1-[R^-(c\sin\psi)]^2},\\
h_2(c)&=\Delta\theta_{DA}(c)=\Delta\theta_{BC}(c),\\
h_3(c)&=\Delta\theta_{AB}(c)=\int_\frac{\pi}{3}^\frac{2\pi}{3}\frac{d\psi}{[R^+(c\sin\psi)]^2-1},
\end{split}
\end{equation}
for $c\geq c_*=\frac{2}{\sqrt{3}}$. Note that $c_*$ is the lowest
energy if the curve connects to a line at a regular triple junction.
We also define $\eta^*=1+2\log c^*=1+\log\frac{4}{3}$. Since $R^-$
is decreasing and $R^+$ is increasing, $h_1$ and $h_3$ are
decreasing functions of $c$. From lemma \ref{lem:left_right_ineq},
we have $h_1>h_3$. Use $T(c)$ to denote the change of angle of a
complete period. For $c\geq c_*$,
\begin{equation}
    T(c)=h_1(c)+2h_2(c)+h_3(c).
\end{equation}
Note from \cite{AL}, $T(c)$ is decreasing and
$\sqrt{2}\pi>T(c)>\pi$.
\begin{lem}
The function $h_1$, $h_2$, $h_3$ is defined on $(c_*,\infty)$ with
the following properties.
\begin{equation}
\lim_{c\to \infty} h_1(c)=\lim_{c\to \infty} h_2(c)=\frac{\pi}{3},
\end{equation}
\begin{equation}
\lim_{c\to \infty} h_3(c)=0.
\end{equation}
\end{lem}
\begin{proof}
This lemma is established in \cite{CG}. We include the proof here
for the completeness. Since $\underset{{s\to\infty}}\lim R^-(s)=0$,
$\underset{s\to\infty}\lim R^+(s)=\infty$,
\begin{equation}
\begin{split}
    \lim_{c\to\infty}\int_\frac{\pi}{3}^\frac{2\pi}{3}\frac{d\psi}{1-[R^-(c\sin\psi)]^2}&=\int_\frac{\pi}{3}^\frac{2\pi}{3}\lim_{c\to\infty}\frac{d\psi}{1-[R^-(c\sin\psi)]^2}=\int_\frac{\pi}{3}^\frac{2\pi}{3}d\psi=\frac{\pi}{3},\\
    \lim_{c\to\infty}\int_\frac{\pi}{3}^\frac{2\pi}{3}\frac{d\psi}{[R^+(c\sin\psi)]^2-1}&=\int_\frac{\pi}{3}^\frac{2\pi}{3}\lim_{c\to\infty}\frac{d\psi}{[R^+(c\sin\psi)]^2-1}=\int_\frac{\pi}{3}^\frac{2\pi}{3}0\cdot d\psi=0.
\end{split}
\end{equation}
Using the result from \cite{AL} about the change of angle of a
complete period, we have
$\underset{c\to\infty}{\lim}(h_1+2h_2+h_3)=\pi$. We can deduce
$\underset{c\to\infty}{\lim}h_2=\frac{\pi}{3}$.
\end{proof}

We are now going to estimate the change of angle which corresponds
to each part of the trajectories. The following estimation as a
lower bound of the potential function $V(k)$ is needed.
\begin{lem}\label{lower2nd}
For $k_{\mathbf{min}}\leq k_1<k\leq 1$, $k_0\leq k_1$, let
$\bar{V}=(1+\frac{1}{k_0})(k-1)^2+H$, where $H$ is chosen such that
$\bar{V}(k_1)=V(k_1)$. We have $V(k)>\bar{V}(k)$ for all
$k\in(k_1,1)$. Therefore, for $k_1<k_2\leq1$,
\begin{equation}
    \Delta\phi=\int_{k_1}^{k_2}\frac{dk}{\sqrt{\eta-V(k)}}>\frac{1}{\sqrt{1+\frac{1}{k_0}}}\left(\sin^{-1}(\frac{1-k_1}{\sqrt{\frac{\eta-H}{1+\frac{1}{k_0}}}})-\sin^{-1}(\frac{1-k_2}{\sqrt{\frac{\eta-H}{1+\frac{1}{k_0}}}})\right).
\end{equation}
For the special case $k_1=k_{\mathbf{min}}$,
$\sqrt{\frac{\eta-H}{1+\frac{1}{k_0}}}=1-k_{\mathbf{min}}$.
\end{lem}
\begin{proof}
For $k_{\mathbf{min}}\leq k_1<k\leq 1$, we have
\begin{equation}
    V'(k)=2(k-\frac{1}{k})=2(1+\frac{1}{k})(k-1)>2(1+\frac{1}{k_0})(k-1)=\bar{V}'(k).
\end{equation}
Use $\bar{V}(k_1)=V(k_1)$, we can deduce $V(k)>\bar{V}(k)$ for all
$k\in(k_1,1)$. We can obtain the estimation of the integral by
direct computation.
\end{proof}

We need a lower bound for $h_1+2h_2$. This quantity plays an
important role when we are excluding some impossible cases.

\begin{thm}\label{h12h2}
For every $\eta\geq \eta_*$(i.e. $c\geq c_*$), we have
$h_1(c)+2h_2(c)>0.7789\pi(>\frac{2\pi}{3})$. Furthermore, if
$\eta\geq\frac{4}{3}$, we have $h_1+2h_2>0.9456\pi$. For
$\eta\geq1.38$, we have $h_1+2h_2>\pi$.
\end{thm}

\begin{proof}
We want to estimate $\Delta\phi$. Let
$\hat{R}(c)=R^+(\frac{\sqrt{3}}{2}c)\geq1$ be the $R$ value at point
$A$. Note that $\eta$ is strictly increasing with respect to
$\hat{R}$.

\emph{Case 1: $\eta\geq1.38$}.
\begin{equation}
    \Delta\phi=\int_{k_{\mathbf{min}}}^{\frac{\sqrt{3}}{2}\hat{R}}\frac{2dk}{\sqrt{\eta-V(k)}}=\int_{k_{\mathbf{min}}}^1\frac{2dk}{\sqrt{\eta-V(k)}}+\int_1^{\frac{\sqrt{3}}{2}\hat{R}}\frac{2dk}{\sqrt{\eta-V(k)}}.
\end{equation}
Define
$L(\hat{R})=\int_{k_{\mathbf{min}}}^1\frac{2dk}{\sqrt{\eta-V(k)}}$
and
$R(\hat{R})=\int_1^{\frac{\sqrt{3}}{2}\hat{R}}\frac{2dk}{\sqrt{\eta-V(k)}}$
to be the contribution of the left side and right side of the
potential function to $\Delta\phi$.

For the left side, from lemma \ref{lower2nd}, we have
\begin{equation}\label{ineq:left}
    L(\hat{R})\geq\int_{k_{\mathbf{min}}}^1\frac{2dk}{\sqrt{\eta-\bar{V}(k)}}=\frac{\pi}{\sqrt{1+\frac{1}{k_{\mathbf{min}}}}}.
\end{equation}
For the right side, let
$\kappa=\sqrt{2\log\hat{R}-2\log\frac{2}{\sqrt{3}}+1}$.
Let$\bar{V}(k)=k^2-2\log\hat{R}+2\log\frac{2}{\sqrt{3}}$ for
$k\in(\kappa,\frac{\sqrt{3}}{2}\hat{R})$ and $\bar{V}(k)=1$ for
$k\in(1,\kappa)$. We have $\bar{V}(k)<V(k)$. The right side is
bounded below by
\begin{equation}\label{ineq:right}
    R(\hat{R})\geq\int_1^{\frac{\sqrt{3}}{2}\hat{R}}\frac{2dk}{\sqrt{\eta-\bar{V}(k)}}=\frac{2\pi}{3}-2\sin^{-1}(\frac{\kappa}{\hat{R}})+2\frac{\kappa-1}{\sqrt{\eta-1}}.
\end{equation}
Note that
\begin{equation}
\begin{split}
\frac{d}{d\hat{R}}&(\frac{\kappa-1}{\sqrt{\hat{R}^2-\kappa^2}}-\sin^{-1}(\frac{\kappa}{\hat{R}}))=\frac{1}{\sqrt{R^2-\kappa^2}^3}(\frac{R^2-\kappa^3}{R}+\frac{\kappa-1}{R}).
\end{split}
\end{equation}
Use
$\frac{d}{d\hat{R}}(\hat{R}^\frac{4}{3}-\kappa^2)=\frac{4}{3}\hat{R}^\frac{1}{3}-2\hat{R}^{-1}$,
the minimum for $\hat{R}^\frac{4}{3}-\kappa^2$ happens when
$\hat{R}^\frac{4}{3}=\frac{3}{2}$. Therefore,
$\hat{R}^\frac{4}{3}-\kappa^2\geq\frac{3}{2}-\frac{3}{2}\log\frac{3}{2}+2\log\frac{2}{\sqrt{3}}-1>0$.
We have $\hat{R}^2>\kappa^3$ and
$\frac{d}{d\hat{R}}(\frac{\kappa-1}{\sqrt{\hat{R}^2-\kappa^2}}-\sin^{-1}(\frac{\kappa}{\hat{R}}))>0$.
$R(\hat{R})$ is bounded below by a function which is increasing when
$\hat{R}$ increases.

The following bound can be obtained by using a scientific calculator
for elementary functions.
\begin{center}
  \begin{tabular}{| c || c | c || c | c |}
    \hline
    $\eta$ & $k_{\mathbf{min}}>$ & Lower bound for $L(\hat{R})$ & $\hat{R}>$ & Lower bound for $R(\hat{R})$ \\ \hline
    1.38 & 0.60 & $0.6123\pi$ & 1.22 & $0.0585\pi$\\ \hline
    1.4 & 0.59 & $0.6091\pi$ & 1.24 & $0.0748\pi$\\ \hline
    1.45 & 0.56 & $0.5991\pi$ & 1.29 & $0.1123\pi$\\ \hline
    1.5 & 0.52 & $0.5848\pi$ & 1.34 & $0.1453\pi$ \\ \hline
    2 & 0.39 & $0.5927\pi$ & 1.64 & $0.2761\pi$\\ \hline
    3 & 0.32 & $0.4246\pi$ & 2.03 & $0.3631\pi$\\ \hline
    4 & 0.13 & $0.3392\pi$ & 2.32 & $0.4035\pi$\\ \hline
    5 & 0.08 & $0.2722\pi$ & 2.56 & $0.4287\pi$\\
    \hline
  \end{tabular}
\end{center}

When $\eta$ increases, $\hat{R}$ increases and $k_{\mathbf{min}}$
decreases. Therefore, the lower bound for the right side increases
and the lower bound for the left side decreases. We have

\begin{center}
  \begin{tabular}{| c | c |}
    \hline
    Range for eta $\eta$ & Lower bound for $\Delta\phi$ \\ \hline
    [1.38, 1.4] & $0.0585\pi+0.6091\pi>\frac{2\pi}{3}$ \\ \hline
    [1.4, 1.45] & $0.0748\pi+0.5991\pi>\frac{2\pi}{3}$ \\ \hline
    [1.45, 1.5] & $0.1123\pi+0.5848\pi>\frac{2\pi}{3}$ \\ \hline
    [1.5, 2] & $0.1453\pi+0.5296\pi>\frac{2\pi}{3}$ \\ \hline
    [2, 3] & $0.2792\pi+0.4246\pi>\frac{2\pi}{3}$ \\ \hline
    [3, 4] & $0.3631\pi+0.3392\pi>\frac{2\pi}{3}$ \\ \hline
    [4, 5] & $0.4035\pi+0.2722\pi>\frac{2\pi}{3}$ \\ \hline
  \end{tabular}
\end{center}

For all $\hat{R}>\frac{5}{2}$, first we compare $k_{\mathbf{min}}$
with $\frac{1}{2\hat{R}^2}$. Note that $V(k_{\mathbf{min}})=\eta$.
Let
\begin{equation}
    U(\hat{R})=\eta-V(\frac{1}{2\hat{R}^2})=\hat{R}^2-6\log\hat{R}+2\log\frac{1}{\sqrt{3}}-\frac{1}{4\hat{R}^4},
\end{equation}
we have
$\frac{d}{d\hat{R}}U=\frac{1}{\hat{R}^5}(2\hat{R}^6-6\hat{R}^4+1)$.
Note that $2\hat{R}^6-6\hat{R}^4+1$ is increasing for
$\hat{R}>\sqrt{2}$, $\frac{d}{d\hat{R}}U(2)>0$, $U$ is increasing
for $\hat{R}>2$. Together with $U(\frac{5}{2})>0$, we can deduce for
$\hat{R}>\frac{5}{2}$, $U>0$ and therefore
$k_{\mathbf{min}}<\frac{1}{2\hat{R}^2}$.

Use
$L(\hat{R})=\int_{k_{\mathbf{min}}}^1\frac{2dk}{\sqrt{\eta-V(k)}}>\frac{2(1-k_{\mathbf{min}})}{\sqrt{\eta-1}}$,
\begin{equation}
\Delta\phi\geq\frac{2\pi}{3}-2\sin^{-1}(\frac{\kappa}{\hat{R}})+\frac{2(\kappa-k_{\mathbf{min}})}{\sqrt{\hat{R}^2-\kappa^2}}\geq2\Bigg(\frac{\pi}{3}-\sin^{-1}(\frac{\kappa}{\hat{R}})+\frac{\kappa-\frac{1}{2\hat{R}^2}}{\sqrt{\hat{R}^2-\kappa^2}}\Bigg).
\end{equation}

Consider $\kappa$ as function of $\hat{R}$. Let
$F(\hat{R})=-\sin^{-1}(\frac{\kappa}{\hat{R}})+\frac{\kappa-\frac{1}{2\hat{R}^2}}{\sqrt{\hat{R}^2-\kappa^2}}$.
We have
\begin{equation}
\begin{split}
F'(\hat{R})=&\frac{1}{\hat{R}\sqrt{\hat{R}^2-\kappa^2}^3}\left(\kappa(1-\kappa^2)+\frac{1}{2}(3-\frac{1}{\hat{R}^2}-\frac{2\kappa^2}{\hat{R}^2})\right)\\
<&\frac{1}{\hat{R}\sqrt{\hat{R}^2-\kappa^2}^3}\left(2\kappa(\log\frac{2}{\sqrt{3}}-\log\hat{R})+\frac{3}{2}\right)<\frac{2}{\hat{R}\sqrt{\hat{R}^2-\kappa^2}^3}\log(\frac{2e^\frac{3}{4}}{\sqrt{3}\hat{R}}).
\end{split}
\end{equation}
For $\hat{R}>\frac{5}{2}>\frac{2e^\frac{3}{4}}{\sqrt{3}}$, we can
deduce that $F'(\hat{R})<0$. Since
$\underset{\hat{R}\to\infty}{\lim}F=0$, we get
$\Delta\phi>\frac{2\pi}{3}$. In each case, we have
$\Delta\theta=\Delta\phi-\Delta\psi>\pi$.

\emph{Case 2: $\frac{4}{3}\leq\eta(\hat{R})<1.38$.} From the proof
of the theorem above, we have $\Delta\phi>L(\hat{R})>0.6123\pi$.
Therefore, $\Delta\theta>0.9456\pi$.

\emph{Case 3: $\eta_*\leq\eta<\frac{4}{3}$.} In this case,
\begin{equation}
1+2\log\frac{2}{\sqrt{3}}\leq\eta(\hat{R})=\hat{R}^2-2\log\hat{R}+2\log\frac{2}{\sqrt{3}}\leq\frac{4}{3}.
\end{equation}
Note that $V(\frac{1}{\sqrt{3}})=\frac{1}{3}+\log3>\frac{4}{3}$. On
the other hand,
$V(\frac{2}{3})=\frac{4}{9}+2\log\frac{9}{4}<1+2\log\frac{2}{\sqrt{3}}$,
we can deduce
$\frac{1}{\sqrt{3}}<k_{\mathbf{min}}<\frac{2}{3}<\frac{5}{6}<\frac{\sqrt{3}}{2}\leq\frac{\sqrt{3}}{2}\hat{R}$.
Therefore, from lemma \ref{lower2nd},
\begin{equation}
\begin{split}
\Delta\phi&>\frac{1}{\sqrt{1+\sqrt{3}}}(\pi-2\sin^{-1}\frac{1-\frac{\sqrt{3}}{2}\hat{R}}{1-k_{\mathbf{min}}})>0.4456\pi(>\frac{\pi}{3}).
\end{split}
\end{equation}

Therefore, $\Delta\phi>0.4456\pi$,
$h_1(c)+2h_2(c)=\Delta\psi+0.4456\pi>0.7789\pi(>\frac{2\pi}{3})$.
\end{proof}

\begin{cor}
The Cisgeminate 3-ray star proposed in the appendix of \cite{MNPS}
does not exist.
\end{cor}
\begin{proof}
By symmetry, the change of angle is $\frac{2\pi}{3}$ for each
5-cell. On the other hand, the change of angle should be $h_1+4h_2$
for the corresponding energy, which is impossible since
$h_1+2h_2>\frac{2\pi}{3}$.
\end{proof}

\begin{prop}\label{thm:half_upper_bound}
For any $c>1$, $\Delta\theta_{MN}(c)<\pi$.
\end{prop}
\begin{proof}
Let $\psi_0$ be the $\psi$ value at $N(c)$. We have
$c=\frac{1}{\sin(\psi_0)}$ and $\eta=1+2\log c=1-2\log \sin\psi_0$.
The curvature corresponds to $(1,\psi_0)$ is $k_0=\sin\psi_0$.
\begin{equation}
\Delta\theta=\Delta\phi-\Delta\psi=\Delta\phi+\pi-2\psi_0.
\end{equation}
We need to estimate $\Delta\phi$. Let $\hat{V}$ be the linear
function passing through
$(k_{\mathbf{min}},\eta)=(k_{\mathbf{min}},1-2\log\sin\psi_0)$ and
$(k_0,V(k_0))=(\sin\psi_0,\sin^2\psi_0-2\log\sin\psi_0)$. Since
$\hat{V}''=0$, $V''>0$, we have $\hat{V}>V$ for all
$k\in(k_{\mathbf{min}},k_0)$.
\begin{equation}
\begin{split}
\Delta\phi&=\int_{k_{\mathbf{min}}}^{k_0}\frac{2dk}{\sqrt{\eta-V(k)}}\leq\int_{k_{\mathbf{min}}}^{k_0}\frac{2dk}{\sqrt{\eta-\hat{V}(k)}}
=4\frac{k_0-k_{\mathbf{min}}}{\cos\psi_0}.
\end{split}
\end{equation}
Note that
\begin{equation}
V(\frac{\sin\psi_0}{\sqrt{e}})=\frac{\sin^2\psi_0}{e}-2\log(\frac{\sin\psi_0}{\sqrt{e}})=\frac{\sin^2\psi_0}{e}-2\log\sin\psi_0+1>\eta.
\end{equation}
We have $\frac{\sin\psi_0}{\sqrt{e}}<k_{\mathbf{min}}$. Therefore,
\begin{equation}
\begin{split}
\Delta\phi\leq4\frac{k_0-k_{\mathbf{min}}}{\cos\psi_0}<2\frac{2(1-\frac{1}{\sqrt{e}})}{\cos\psi_0}\sin\psi_0.
\end{split}
\end{equation}
For $\psi_0<0.21\pi<\cos^{-1}(2(1-\frac{1}{\sqrt{e}}))$, we have
$\Delta\phi\leq2\sin\psi_0\leq2\psi_0$.

For $0.21\pi\leq\psi_0<\frac{\pi}{2}$, we can improve the lower
bound for $k_{\mathbf{min}}$. Since
\begin{equation}
V(\frac{\sin\psi_0}{\sqrt{2}})=\frac{\sin^2\psi_0}{2}-2\log(\frac{\sin\psi_0}{\sqrt{2}})=\frac{\sin^2(\frac{\pi}{5})}{2}-2\log\sin\psi_0+\log2>\eta,
\end{equation}
$k_{\mathbf{min}}>\frac{\sin\psi_0}{\sqrt{2}}$ in this case. Let
$\hat{V}(k)=2(k-1)^2+L$ where $L$ is chosen such that
$\bar{V}(k_{\mathbf{min}})=\eta$. We have $\hat{V}>V$ in our
interval of integration since $\hat{V}''<V''$,
$\hat{V}(k_{\mathbf{min}})=V(k_{\mathbf{min}})$ and
$\hat{V}'(1)=V'(1)=0$.
\begin{equation}
\begin{split}
\Delta\phi&\leq\int_{k_{\mathbf{min}}}^{\sin\psi_0}\frac{2dk}{\sqrt{\eta-L-2(k-1)^2}}
=\sqrt{2}\cos^{-1}(\frac{1-\sin\psi_0}{1-k_{\mathbf{min}}})<\sqrt{2}\cos^{-1}(\frac{1-\sin\psi_0}{1-\frac{\sin\psi_0}{\sqrt{2}}}).
\end{split}
\end{equation}
Note that we use the substitution
$-\sqrt{\frac{\eta-L}{2}}\cos\alpha=k-1$.

Let $f(\psi_0)=\sin\psi_0+\cos\sqrt{2}\psi_0(1-\frac{\sin\psi_0
}{\sqrt{2}})$, we want to show that $f<1$ for $\frac{\pi}{5}\leq
k<\frac{\pi}{2}$. We have $f<1$ for
$\psi_0\geq\frac{\pi}{2\sqrt{2}}$.
\begin{center}
  \begin{tabular}{| c | c |}
    \hline
    Range for $\psi_0$ & Upper bound for $f(\psi_0)$ \\ \hline
    $[0.3\pi, \frac{\pi}{2\sqrt{2}})$ & $\sin\frac{\pi}{2\sqrt{2}}+\cos(\sqrt{2}\cdot0.3\pi)(1-\frac{\sin 0.3 \pi}{\sqrt{2}})<0.9969$ \\ \hline
    $[0.27\pi, 0.3\pi)$ & $\sin0.3\pi+\cos(\sqrt{2}\cdot0.27\pi)(1-\frac{\sin 0.27 \pi}{\sqrt{2}})<0.9794$ \\ \hline
    $[0.24\pi, 0.27\pi)$ & $\sin0.27\pi+\cos(\sqrt{2}\cdot0.24\pi)(1-\frac{\sin 0.24 \pi}{\sqrt{2}})<0.9996$ \\ \hline
    $[0.22\pi, 0.24\pi)$ & $\sin0.24\pi+\cos(\sqrt{2}\cdot0.22\pi)(1-\frac{\sin 0.22 \pi}{\sqrt{2}})<0.9917$ \\ \hline
    $[0.21\pi, 0.22\pi)$ & $\sin0.22\pi+\cos(\sqrt{2}\cdot0.21\pi)(1-\frac{\sin 0.21 \pi}{\sqrt{2}})<0.9748$ \\ \hline
  \end{tabular}
\end{center}

This is equivalent to $\Delta\phi<2\psi_0$.

\end{proof}

If we set $\psi_0=\frac{\pi}{3}$, $A(c_*)=N(c_*)=D(c_*)$ and
$B(c_*)=M(c_*)=D(c_*)$ in the proposition, we have the following
corollary.
\begin{cor}\label{cor:h1upper}
The upper bound of $h_1$ is given by
\begin{equation}
h_1(c_*)<\pi.
\end{equation}
\end{cor}

\section{The possible topology of a regular shrinker with 2 closed regions}

Now, we turn our attention to the topology of a regular shrinker
with possibly more than 2 closed regions. Remove all the rays from
such a regular shrinker and consider it as a graph $G$ with $E$
edges and $V$ vertices.

\begin{lem} \label{lem:topology_network}
For any regular shrinker with at least one closed region, let $F_i$
be the closed regions enclosed by the network. Then
$\cup_{i}\mathbf{cl}(F_i)$ is star-shaped with respect to the origin
$O$.
\end{lem}
\begin{proof}
From the graph $G$ defined above, define
$\rho:\mathcal{I}\subset\mathbb{S}^1\to \mathbb{R}$ by
\begin{equation}
    \rho(t)=\underset{\{x\in\mathbb{R}^+|xt\in G\}}{\max} x,
\end{equation}where $\mathcal{I}$ is the maximal subset of $\mathbb{S}^2$ such that $\rho$ can be defined. Since $G$ is a compact set, if the set $\{x\in\mathbb{R}^+|xt\in G\}$ is not empty, we can get the maximum value.

For any $t\in\mathcal{I}$, $\rho(t)t\in G$. If $\rho(t)t$ is a
vertex of $G$, since the edges intersection at $\rho(t)t$ and the
angle between the curves are $\frac{2\pi}{3}$, there should be at
least 1 curve going  clockwise and 1 curve going counterclockwise
from $\rho(t)t$. Therefore, there should be a neighborhood of $t$
which is contained in $\mathcal{I}$.  If $\rho(t)t$ lies on an edge
of $G$, this edge cannot be a line segment since there exist one
endpoint of the line segment corresponds to the same
$t\in\mathbf{S}^1$ with larger distance from the origin. Therefore,
it muse lies on a segment of a nondegenerate AL-curve and there
should be a neighborhood of $t$ which is contained in $\mathcal{I}$.
$\mathcal{I}$ is open in $\mathbb{S}^1$.

For any sequence $t_i\in\mathcal{I}$, $t_i\to t$, $\rho(t_i)t_i$ is
a sequence in $G$. Since $G$ is compact, there must be a limit point
$x$ of $\rho(t_i)t_i$ in $G$. Since there are only finitely many
nondegenerate AL-curves in $G$ and each $\rho(t_i)t_i$ lies on
either a segment of a nondegenerate AL-curve or an endpoint of a
segment of a nondegenerate AL-curve. $\rho(t_i)$ is bounded away
from 0. Therefore $|x|>0$, $x=|x|t$ and $t\in\mathcal{I}$.
$\mathcal{I}$ is closed in $\mathbb{S}^1$. Since $\mathcal{I}$ is
nonempty, we have $\mathcal{I}=\mathbb{S}^1$.

Note that $\rho$ is upper semi-continuous, $\underset{t\to
t_0}{\lim\sup} \rho(t)\leq \rho(t_0)$. For any $t_0\in \mathbb S^1$,
$\rho(t_0)t_0$ is a vertex or belongs to a non-degenerate AL-curve.
Again, there exists a neighborhood of $\rho(t_0)t_0$ in $G$. We can
find a sequence $P_i\in G$ in the neighborhood such that it
converges to $\rho(t_0)t_0$ and $t_i=\frac{P_i}{|P_i|}\neq t_0$,
$t_i$ converges to $t_0$. We obtain
\[
\rho(t_0)=\underset{P_i\to \rho(t_0)t_0}\lim |P_i|
\leq \underset{t_i\to t_0}\liminf \rho(t)
\leq \underset{t_i\to t_0}\limsup \rho(t)\leq \rho (t_0).
\]
Therefore, $\rho(t_0)=\underset{t\to t_0}{\lim} \rho(t)$ and  $\rho$
is continuous on $\mathbb S^1$. Let $\Gamma(t)=\rho(t)t$ and $F$ be
the finite region enclosed by the curve $\Gamma(t)$. Note that the
origin $O$ belongs to $F$ because $\mathcal{I}=\mathbb S^1$. Since
$\Gamma\subset G_0$ and $G_0\subset\mathbf{cl}(F)$, we obtain
$\underset{i}\cup\mathbf{cl}(F_i)=\mathbf{cl}(F)$ is star-shaped
with respect to the origin.
\end{proof}

Now, we turn our attention to the topology of regular shrinker with
2 closed regions.
\begin{thm}\label{topology}
The topology of a regular shrinker with 2 closed regions must be a
$\Theta$-shaped network with possibly multiple rays attached to the
outer curves.
\end{thm}
\begin{proof}
Use lemma \ref{lem:topology_network}, the 2 closed regions share at
least an edge. If they share more than 1 edge, we obtain either
there are more than 2 closed regions or some multipoints are not
triple junctions. It is impossible.

We need to exclude the case that one of the regions is a 1-cell
surrounded by another region with 4-cell or 5-cell. Let $\gamma_1$
be the boundary of the 1-cell and $\gamma_2$ be the piecewise smooth
curve which is the boundary of
$\mathbf{cl}(F_1)\cup\mathbf{cl}(F_2)$, where $F_i$ are the closed
regions of the network. There can be at most one ray attached to
$\gamma_2$. Using lemma \ref{same_energy}, the energies of all
smooth AL-curves of $\gamma_2$ are the same. Let $c_i$ be the energy
of $\gamma_i$. Since $T(c)<\sqrt{2}\pi$ and the change of angle on
$\gamma_1$ is $2\pi$, the curve $\gamma_1$ consists more than a
complete period. The $R$ value of $\gamma_1$ must achieve the
maximum $R^+(c_1)$ and the minimum $R^-(c_1)$. Since $\gamma_1$ is
included in the region enclosed by $\gamma_2$, there exists a value
$R$ on $\gamma_2$ with $R>R^+(c_1)$. Therefore, we have $c_2>c_1$.

If there is no ray attached to $\gamma_2$, since the change of angle
is $2\pi>\sqrt{2}\pi>T(c_2)$, it must achieve the minimum value
$R^-(c_2)$. Since $c_2>c_1$, $\gamma_1$ and $\gamma_2$ intersect. We
obtain a contradiction. If there is a ray attached to $\gamma_2$,
suppose $\gamma_2$ achieve the minimum, we can argue as above.
Suppose $\gamma_2$ does not achieve the minimum, the two curves must
correspond to $BC$ arc and $DA$ arc on the trajectory. The change of
angle is less than $T(c_2)<\sqrt{2}\pi<2\pi$ and we obtain a
contradiction.
\end{proof}

From the theorem, the topology of the network is a $\Theta$ with
rays attached to either side. From lemma \ref{same_energy}, any 2
AL-curves which share a triple junction with a ray have the same
energy. Therefore, for a regular shrinker with 2 closed regions,
there are at most 3 piecewise smooth curves with different energy.
They correspond to the 3 arcs of the original $\Theta$ network.
Without loss of generality, we can rotate the network so that the
line connecting 2 triple junctions of the original $\Theta$ is
parallel to the $x$-axis and the origin is not contained in the
upper closed region. We call the triple junction on the right as the
starting point and the other triple junction of the original
$\Theta$ network as the ending point. We call the inner curve of the
$\Theta$ network $\gamma_{\mathbf{in}}$. Aside from
$\gamma_{\mathbf{in}}$, there are 2 piecewise smooth curves
consisting of AL-curves which goes from the starting point to the
ending point. We call them $\gamma_{\mathbf{up}}$,
$\gamma_{\mathbf{down}}$ depending on whether they go in the
counterclockwise direction or the clockwise direction from the
starting point to the ending point.

Let $R_{\mathbf{start}}$, $R_{\mathbf{end}}$ be the $R$ value for
the starting point and the ending point respectively.
\begin{prop}\label{prop:sym_start_end}
Let $\psi_{\mathbf{start,up}}$, $\psi_{\mathbf{end,up}}$,
$\psi_{\mathbf{start,in}}$, $\psi_{\mathbf{end,in}}$, be the
corresponding $\psi$ at the starting point or the ending point of
$\gamma_{\mathbf{up}}$, $\gamma_{\mathbf{in}}$ respectively. Then
$\psi_{\mathbf{start,up}}+\psi_{\mathbf{end,up}}=\pi$,
$\psi_{\mathbf{start,in}}+\psi_{\mathbf{end,in}}=\pi$, and
$K(R_{\mathbf{start}})=K(R_{\mathbf{end}})$.
\end{prop}
\begin{proof}
For a regular triple junction,
$\psi_{\mathbf{start},\mathbf{up}}+\frac{2\pi}{3}=\psi_{\mathbf{start},\mathbf{in}}$.
Similarly,
$\psi_{\mathbf{end},\mathbf{up}}-\frac{2\pi}{3}=\psi_{\mathbf{end},\mathbf{in}}$.
Compute the energy of $\gamma_{\mathbf{up}}$ and
$\gamma_{\mathbf{in}}$ gives
\begin{equation}
\begin{split}
c_{\mathbf{up}}&=\frac{K(R_{\mathbf{start}})}{\sin(\psi_{\mathbf{start},\mathbf{up}})}=\frac{K(R_{\mathbf{end}})}{\sin(\psi_{\mathbf{end},\mathbf{up}})},\\
c_{\mathbf{in}}&=\frac{K(R_{\mathbf{start}})}{\sin(\psi_{\mathbf{start},\mathbf{in}})}=\frac{K(R_{\mathbf{end}})}{\sin(\psi_{\mathbf{end},\mathbf{in}})}.\\
\end{split}
\end{equation}
We obtain
\begin{equation}
\frac{\sin(\psi_{\mathbf{start},\mathbf{up}})}{\sin(\psi_{\mathbf{start},\mathbf{up}}+\frac{2\pi}{3})}=\frac{\sin(\psi_{\mathbf{end},\mathbf{up}})}{\sin(\psi_{\mathbf{end},\mathbf{up}}-\frac{2\pi}{3})}.
\end{equation}
We omit the subscript "up" in the next equation. The equation is
equivalent to
\begin{equation}
\sin(\psi_{\mathbf{start}})[-\frac{1}{2}\sin(\psi_{\mathbf{end}})-\frac{\sqrt{3}}{2}\cos(\psi_{\mathbf{end}})]=\sin(\psi_{\mathbf{end}})[-\frac{1}{2}\sin(\psi_{\mathbf{start}})+\frac{\sqrt{3}}{2}\cos(\psi_{\mathbf{start}})]
\end{equation}
After combining some terms, we have
\begin{equation}
0=\frac{\sqrt{3}}{2}[\sin(\psi_{\mathbf{start}})\cos(\psi_{\mathbf{end}})+\sin(\psi_{\mathbf{end}})\cos(\psi_{\mathbf{start}})=\frac{\sqrt{3}}{2}\sin(\psi_{\mathbf{start}}+\psi_{\mathbf{end}}).
\end{equation}
Therefore, $\psi_{\mathbf{start}}+\psi_{\mathbf{end}}=\pi$.
\end{proof}

If we move along $\gamma_{\mathbf{down}}$ from the starting point to
the ending point, we are moving clockwisely. In order to use the
setting for counterclockwisely oriented AL-curve in section 2. We
use clockwise direction as the positive direction for the $\theta$,
$\phi$, $\psi$ value related to $\gamma_{\mathbf{down}}$. In this
setting, we have
$\psi_{\mathbf{start},\mathbf{down}}=\frac{2\pi}{3}-\psi_{\mathbf{start},\mathbf{up}}$.

Define $\theta_{\mathbf{up}}$, $\theta_{\mathbf{in}}$,
$\theta_{\mathbf{down}}$ be the total change of angle for the curves
respectively. Note that $\theta_{\mathbf{down}}$ is measured
clockwisely. We have
\begin{equation}
\theta_{\mathbf{up}}=\theta_{\mathbf{in}}=2\pi-\theta_{\mathbf{down}}.
\end{equation}
From the symmetry, the $\psi$ value at the starting point gives
suffice information for the $\psi$ value at the ending point. From
now on, use $\psi_{\mathbf{up}}$, $\psi_{\mathbf{in}}$,
$\psi_{\mathbf{down}}$ to describe the $\psi$ value for each curve
at the starting point for simplicity.
$\psi_{\mathbf{in}}=\psi_{\mathbf{up}}+\frac{2\pi}{3}$,
$\psi_{\mathbf{down}}=\frac{2\pi}{3}-\psi_{\mathbf{up}}$.

\section{The cell which does not contain the origin}
Since there are 2 closed regions, at least one of them does not
contain the origin in the interior. We can follow the argument in
\cite{BHM2} to show it must be a 4-cell. The following theorem
concerning 2-cell is established in \cite{BHM2}.
\begin{thm}[\cite{BHM2}]
In a self-similarly shrinking network moving by curvature, there are
no 2-cells without the origin inside.
\end{thm}

If the closed region which does not contain the origin is a 3-cell,
a 4-cell, or a 5-cell, we have the following lemma.
\begin{lem}\label{lem:no_period_up}
We have the following result concerning $\gamma_{\mathbf{up}}$ and
$\gamma_{\mathbf{in}}$.
\begin{enumerate}
\item If $\gamma_{\mathbf{up}}$ passes through the point corresponding to $(R^-(c_{\mathbf{up}}),\frac{\pi}{2})$, we have $c_{\mathbf{in}}>c_{\mathbf{up}}$.
\item It is impossible for $\gamma_{\mathbf{up}}$ has a complete period on its trajectory.
\end{enumerate}
\end{lem}
\begin{proof} We need part 1 to establish part 2.

(1) If $\gamma_{\mathbf{up}}$ pass through the point
$(R^-(c_{\mathbf{up}}),\frac{\pi}{2})$ on the trajectory, since
$\gamma_{\mathbf{in}}$ lies inside, if we connect the origin with
the point corresponding to $(R^-(c_{\mathbf{up}}),\frac{\pi}{2})$
with a line segment, the line segment must intersect
$\gamma_{\mathbf{in}}$. From this, we have
$R^-(c_{\mathbf{in}})<R^-(c_{\mathbf{up}})$, this is equivalent to
$c_{\mathbf{in}}>c_{\mathbf{up}}$.

(2) If $\gamma_{\mathbf{in}}$ is nondegenrate, since
$\psi_{\mathbf{in}}>\frac{2\pi}{3}$, the starting point of
$\gamma_{\mathbf{in}}$ lies on the $BC$ arc of the trajectory, the
ending point of $\gamma_{\mathbf{in}}$ lies on the $DA$ arc of the
trajectory. Assume $\gamma_{\mathbf{in}}$ passes through  the point
corresponding to  $(R^+(c_{\mathbf{in}}),\frac{\pi}{2})$, the point
with largest $R$ on the trajectory, since
$c_{\mathbf{in}}>c_{\mathbf{up}}$,
$R^+(c_{\mathbf{in}})>R^+(c_{\mathbf{up}})$, $\gamma_{\mathbf{in}}$
and $\gamma_{\mathbf{up}}$ must intersect and we get a
contradiction. Therefore, on the phase plane, $\gamma_{\mathbf{in}}$
only achieve the part from point $B$ to point $A$ on its trajectory.
\begin{equation}
\theta_{\mathbf{in}}\leq(h_1+2h_2)(c_{\mathbf{in}})<T(c_{\mathbf{in}})<T(c_{\mathbf{up}}).
\end{equation}
If $\gamma_{\mathbf{up}}$ has a complete period on its trajectory,
$\theta_{\mathbf{up}}>T(c_{\mathbf{up}})$ and this contradict
$\theta_{\mathbf{up}}=\theta_{\mathbf{in}}$. If
$\gamma_{\mathbf{in}}$ is degenerate, we have
$\theta_{\mathbf{in}}=\pi<T(c_\mathbf{up})<\theta_{\mathbf{up}}$ and
we get a contradiction.

\end{proof}

To eliminate the possibility that this cell is a 3-cell, we need the
following lemma from \cite{BHM2}.
\begin{lem}[\cite{BHM2}]\label{pi2}
Let $\gamma$ be a shrinking curve, parametrized counterclockwise by
arc length, with positive curvature and let $(s_0,s_1)$ be an
interval where $R(s)$ is increasing. If $R_s(s_0)\geq\frac{1}{2}$,
namely,  $\psi(s_0)\leq\frac{\pi}{3}$, then
\begin{equation}
\int_{s_0}^{s_1}\frac{d\theta}{ds} ds<\frac{\pi}{2}.
\end{equation}
Similarly, if $R(s)$ is decreasing on $(s_0,s_1)$ and
$\frac{dR}{ds}(s_1)\leq-\frac{1}{2}$, namely,
$\psi(s_1)\geq\frac{2\pi}{3}$, then the same conclusion holds. This
is equivalent to $2h_2+h_3\leq\pi$.
\end{lem}
\begin{thm}
The upper cell cannot be a 3-cell.
\end{thm}
\begin{proof}
For the 3-cell we are studying, label the triple junction connected to the ray as $P$. Use the notation from the previous section, we have the starting point $S$ and the ending point $E$. 
The curve $\gamma_{\mathbf{in}}$ goes directly from $S$ to $E$. The
piecewise smooth curve $\gamma_{\mathbf{up}}$ goes from $S$ to $P$
and then from $P$ to $E$. We name the part from $S$ to $P$ as
$\gamma_{\mathbf{up}}^1$ and the second part from $P$ to $E$ as
$\gamma_{\mathbf{up}}^2$.

If $\gamma_{\mathbf{up}}$ passes through the point
$(R^+(c_{\mathbf{up}}),\frac{\pi}{2})$ on the trajectory, without
loss of generality, assume it happens on $\gamma_{\mathbf{up}}^1$.
On the phase plane, $\gamma_{\mathbf{up}}^1$ starts at a point on
the $DA$ arc, passes through $(R^+(c_{\mathbf{up}}),\frac{\pi}{2})$
on the trajectory and at the ending P of $\gamma_{\mathbf{up}}^1$,
the corresponding point must be either $D$ or $A$. If it ends at
point $A$ on the phase plane, we have a complete period of
$(R,\psi)$ when traversing $\gamma_1$. If it ends at point $D$,
consider the curve $\gamma_{\mathbf{up}}^2$, the starting point on
the phase plane is $C$ and it ends somewhere between $B$ and $C$.
Therefore, $\gamma_{\mathbf{up}}$ covers a complete period of the
trajectory$(R,\psi)$. This is impossible from lemma
\ref{lem:no_period_up}.

From the previous part, $\gamma_{\mathbf{up}}^1$ and
$\gamma_{\mathbf{up}}^2$ do not pass through
$(R^+(c_{\mathbf{up}}),\frac{\pi}{2})$ on the phase plane. $R$ is
strictly increasing on $\gamma_{\mathbf{up}}^1$ and is strictly
decreasing on $\gamma_{\mathbf{up}}^2$. Now, we separate into 2
cases. From the previous section, we choose the coordinate such that
the line passes through point $S$ and point $E$ is parallel to the
$x$-axis. Let $y=m$ be the equation for this line.

\emph{Case 1: $m>0$}. Let $s_1^*$ be the arc length parameter at the
start of $\gamma_{\mathbf{up}}^1$. Since $\gamma_{\mathbf{up}}^1$,
$\gamma_{\mathbf{in}}$ are above $L$, the angle between
$\frac{d\gamma_{\mathbf{up}}^1}{ds}(s_1^*)$ and $(1,0)$ is less than
or equal to $\frac{\pi}{3}$. Using the ODE describing the
self-shrinking curve, we can extend $\gamma_{\mathbf{up}}^1$ to
$s<s_1^*$. This curve must intersect positive $x$-axis at some
$\tilde{s}_1$. Since the curvature is positive, the angle between
$\gamma_{\mathbf{up}}^1(\tilde{s}_1)$ and $(1,0)$ is less than or
equal to $\frac{\pi}{3}$. Similarly, let $s_2^*$ be the arc length
parameter at the end of $\gamma_{\mathbf{up}}^2$. We can extend
$\gamma_{\mathbf{up}}^2$ beyond $s_2^*$ to intersect negative
$x$-axis at $\gamma_{\mathbf{up}}^2(\tilde{s}_2)$ and the angle
between $\frac{d\gamma_{\mathbf{up}}^2}{ds}(\tilde{s}_2)$ and
$(1,0)$ is less than or equal to $\frac{\pi}{3}$. The change of
angle on the extended curve from
$\gamma_{\mathbf{up}}^1(\tilde{s}_1)$ to
$\gamma_{\mathbf{up}}^2(\tilde{s}_2)$ is exactly $\pi$. There should
be at least one extended curve with change of angle greater then or
equal to $\frac{\pi}{2}$. Without loss of generality, assume
extended $\gamma_{\mathbf{up}}^1$ has this property. Since from the
starting point to the end point of extended
$\gamma_{\mathbf{up}}^1$, $R$ is monotonically increasing, we obtain
a contradiction by lemma \ref{pi2}.

\emph{Case 2: $m\leq 0$}. Either the change of angle $\theta$ of
$\gamma_{\mathbf{up}}^1$  or $\gamma_{\mathbf{up}}^2$ is greater
than or equal to $\frac{\pi}{2}$ since their summation must exceed
$\pi$. Without loss of generality, we can assume
$\gamma_{\mathbf{up}}^1$ satisfies condition. Note that
$\psi_{\mathbf{up}}\leq\frac{\pi}{3}$ at the start of
$\gamma_{\mathbf{up}}^1$. Since from the starting point to the end
point of $\gamma_{\mathbf{up}}^1$, $R$ is monotonically increasing,
we obtain a contradiction by lemma \ref{pi2}.
\end{proof}

\begin{thm}
If the upper cell is a 4-cell, the curve $\gamma_{\mathbf{up}}$ on
the phase plane must be $SA\rightarrow BA\rightarrow BE$. We also
have $c_{\mathbf{in}}>c_{\mathbf{up}}$ and
$\frac{\pi}{6}<\psi_{\mathbf{up}}\leq\frac{\pi}{3}$.
\end{thm}
\begin{proof}
On $\gamma_{\mathbf{up}}$, the starting point lies on the $DA$ arc
and the ending point lies on the $BC$ arc. The only possibility for
$\gamma_{\mathbf{up}}$ does not have a complete period on the
trajectory is that all the triple junction goes from $A$ to $B$.
Note that the curve from a triple junction to another triple
junction must pass through $(R^-(c_{\mathbf{up}}),\frac{\pi}{2})$,
we have  $c_{\mathbf{in}}>c_{\mathbf{up}}$. Use
$c_{\mathbf{in}}=\frac{K(R_{\mathbf{start}})}{\psi_{\mathbf{in}}}$,
$c_{\mathbf{up}}=\frac{K(R_{\mathbf{start}})}{\psi_{\mathbf{up}}}$
and $\psi_{\mathbf{in}}=\psi_{\mathbf{up}}+\frac{2\pi}{3}$, we have
$\frac{\pi}{6}<\psi_{\mathbf{up}}\leq\frac{\pi}{3}$.
\end{proof}

\begin{thm}
The upper cell cannot be a 5-cell.
\end{thm}
\begin{proof}
Consider the curve from a triple junction to another triple
junction. It starts at either $B$ or $C$ and it end at either $D$ or
$A$ on the trajectory. It must pass through
$(R^-(c_{\mathbf{up}}),\frac{\pi}{2})$, we have
$c_{\mathbf{in}}>c_{\mathbf{up}}$. Again, on $\gamma_{\mathbf{up}}$,
the starting point lies on the $DA$ arc and the ending point lies on
the $BC$ arc. The only possibility for $\gamma_{\mathbf{up}}$ does
not have a complete period on the trajectory is $SA\rightarrow
BD\rightarrow CA\rightarrow BE$ or $SA\rightarrow BA\rightarrow
BA\rightarrow BE$. Therefore,
$\theta_{\mathbf{up}}\geq(2h_1+2h_2)(c_{\mathbf{up}})$. Using
$h_1>h_3$, the change of angle is greater than $T(c_{\mathbf{up}})$.
Use the argument as in the proof of lemma \ref{lem:no_period_up}, we
can conclude that there does not exist such 5-cell.
\end{proof}
\begin{rmk}
The theorems about the upper cells are not restrict to a
$\Theta$-shaped network. They can be applied to any closed region in
a regular shrinker with only 1 edge connected to another closed
region and without the origin inside.
\end{rmk}
\begin{rmk}
From the theorem above, we can conclude the regular shrinker with
the topology of Cisgeminate 4-ray star proposed in the appendix of
\cite{MNPS} does not exist.
\end{rmk}

\section{The structure of the lower curve}

For a regular shrinker, any closed region has at most 5 edges.
Furthermore, for a $\Theta$-shaped network with lines, there is at
least one closed region which does not enclose the origin. From the
previous section, such closed region must be a 4-cell. Now, there
are 4 topology type remain possible: a 4-cell together with either a
5-cell, a 4-cell, a 3-cell, a 2-cell. From now on, we use $S$, $E$
to denote the starting point and the ending point on the trajectory
respectively.

\begin{prop}\label{prop:energy_comparison}
For the energy of the 3 curves, we have
$c_{\mathbf{in}}>c_{\mathbf{up}}\geq c_{\mathbf{down}}$.
\end{prop}
\begin{proof}
We have $c_{\mathbf{in}}>c_{\mathbf{up}}$ and
$\frac{\pi}{6}<\psi_{\mathbf{up}}\leq\frac{\pi}{3}$ from the
previous section. Therefore,
$\frac{\pi}{3}\leq\psi_{\mathbf{down}}<\frac{\pi}{2}$. From
$\psi_{\mathbf{up}}\leq\frac{\pi}{3}\leq\psi_{\mathbf{down}}$, we
obtain $c_{\mathbf{up}}\geq c_{\mathbf{down}}$.
\end{proof}
\begin{prop}\label{inout}
If $R_{\mathbf{start}}<1$ or $R_{\mathbf{end}}<1$, the change of
angle $\theta_{\mathbf{in}}\leq\pi$.
\end{prop}
\begin{proof}
For the special case $c_{\mathbf{in}}=\infty$, we have
$\theta_{\mathbf{in}}=\pi$. Otherwise, when
$R_{\mathbf{start}}=R_{\mathbf{end}}<1$ since at the start of
$\gamma_{\mathbf{in}}$, $\frac{5}{6}\pi<\psi_{\mathbf{in}}\leq\pi$
and at the end of $\gamma_{\mathbf{in}}$,
$\psi=\pi-\psi_{\mathbf{in}}$ and it cannot contain a complete loop
of the trajectory, the part of the trajectory is less than the
change of angle going from $M$ to $N$ counterclockwisely on the
trajectory. Therefore, it is less than $\Delta\theta_{MN}$.

If either $R_{\mathbf{start}}>1$ or $R_{\mathbf{end}}>1$, without
loss of generality, assume $R_{\mathbf{start}}<1<R_{\mathbf{end}}$.
We want to compare the change of angle from $M$ to $N$ and the
change of angle from $S$ to $E$. Using lemma
\ref{lem:left_right_ineq}, we have
\begin{equation}
\begin{split}
\Delta\theta_{MS}(c)&=\int_{\psi_{\mathbf{in}}}^{\psi_{\mathbf{max}}}\frac{d\psi}{1-[R^-(c\sin\psi)]^2}\geq\int_{\psi_{\mathbf{in}}}^{\psi_{\mathbf{max}}}\frac{d\psi}{[R^+(c\sin\psi)]^2-1}\\
&=\int^{\pi-\psi_{\mathbf{in}}}_{\psi_{\mathbf{min}}}\frac{d\psi}{[R^+(c\sin\psi)]^2-1}=\Delta\theta_{NE}(c).
\end{split}
\end{equation}
Therefore, $\Delta\theta_{SE}(c)<\Delta\theta_{MN}(c)$. From theorem
\ref{thm:half_upper_bound},
$\Delta\theta_{SE}(c)<\Delta\theta_{MN}(c)<\pi$.
\end{proof}

\begin{lem}\label{noperiod}
It is impossible for $\gamma_{\mathbf{down}}$ to have a complete
period of the trajectory.
\end{lem}

\begin{proof}
If $\eta_{\mathbf{up}}<1.38$, from proposition
\ref{prop:energy_comparison}, we have
$\eta_{\mathbf{down}}\leq\eta_{\mathbf{up}}<1.38$. Note that
$V(0.6)>1.38$, therefore, $0.6<k_{\mathbf{min}}$ when $\eta<1.38$.
Use lemma \ref{lower2nd}, we have
\begin{equation}
   \int_{k_{\mathbf{min}}}^1\frac{2dk}{\sqrt{\eta-V(k)}}>0.6123\pi.
\end{equation}
On the other hand, the potential for $k>1$ is bounded below by
$2(x-1)^2+H$, where $H$ is chosen such that this parabola pass
through $(k_{\max},V(k_{\max}))$.  We have the lower bound
$\frac{\pi}{\sqrt{2}}$. The period is bounded below by
\begin{equation}
   T(c_{\mathbf{down}})> 0.6123\pi+\frac{\pi}{\sqrt{2}}>1.3194\pi.
\end{equation}
We obtain $\theta_{\mathbf{down}}>T(c_{\mathbf{down}})>1.3194\pi$.
On the other hand, from theorem \ref{h12h2},
$\theta_{\mathbf{up}}>(h_1+2h_2)(c_{\mathbf{up}})>0.7789\pi$.

If $\eta_{\mathbf{up}}\geq1.38$, from theorem \ref{h12h2},
$\theta_{\mathbf{up}}>(h_1+2h_2)(c_{\mathbf{up}})>\pi$. From the
result of Abresch and Langer \cite{AL},
$\theta_{\mathbf{down}}>T(c_{\mathbf{down}})>\pi$. In both case, it
is impossible since
$\theta_{\mathbf{up}}+\theta_{\mathbf{down}}=2\pi$.
\end{proof}

We deal with the case that the bottom cell is a 2-cell first. This
is quite different from the 3-cell, 4-cell, 5-cell cases.
\begin{thm}
It is impossible for the bottom cell to be a 2-cell.
\end{thm}
\begin{proof}
In this case, the trajectory for $\gamma_{\mathbf{down}}$ in the
phase plane may not touch $\psi=\frac{\pi}{3}$ and
$\psi=\frac{2\pi}{3}$. Therefore, the point A, B, C, D may be
undefined on the trajectory and the method of expressing angles in
terms of $h_1$, $h_2$, and $h_3$ may not be applicable. For this
case, we only use the point $M$, $N$. Since there are no triple
junctions on $\gamma_{\mathbf{down}}$, we have
$\theta_{\mathbf{down}}=\Delta\theta_{SE}$. We separate into 3
cases.

When $R_{\mathbf{start}}=R_{\mathbf{end}}<1$, if
$\eta_{\mathbf{up}}\geq 1.38$, using theorem \ref{h12h2} and
proposition \ref{inout}, we have
$\pi\geq\theta_{\mathbf{in}}=\theta_{\mathbf{up}}>(h_1+2h_2)(c_{\mathbf{up}})>\pi$.
It is impossible. If  $\eta_{\mathbf{up}}<1.38$, we have
$c_{\mathbf{up}}=e^{\frac{\eta_{\mathbf{up}}-1}{2}}<e^{0.19}$. Since
$\sin(\psi_{\mathbf{up}})=\frac{K(R_{\mathbf{start}})}{c_{\mathbf{up}}}\geq
e^{-0.19}$, we have $\psi_{\mathbf{up}}\geq 0.3099\pi$ and
$\psi_{\mathbf{down}}=\frac{2\pi}{3}-\psi_{\mathbf{up}}\leq
0.3568\pi$,
\begin{equation}
\begin{split}
\theta_{\mathbf{down}}+\theta_{\mathbf{in}}&\leq T(c_{\mathbf{down}})-\int_{\psi_{\mathbf{down}}}^{\pi-\psi_{\mathbf{down}}}\frac{1}{1-R^-(c_{\mathbf{down}}\sin\psi)^2}d\psi+\pi\\
&\leq\sqrt{2}\pi-\frac{\pi-2\psi_{\mathbf{down}}}{1-R^-(c_{\mathbf{down}})^2}+\pi\leq\sqrt{2}\pi-\frac{\pi-2\times 0.3568\pi}{1-0.6^2}+\pi<2\pi,
\end{split}
\end{equation}
the last inequality comes from $R^-(c_{\mathbf{down}})\geq
R^-(c_{\mathbf{up}})\geq R^-(e^{0.19})>0.6$. Therefore, there does
not exist a bottom 2-cell in this case.

When $R_{\mathbf{start}}<1<R_{\mathbf{end}}$, using the symmetry of
the trajectory with respect to $\psi=\frac{\pi}{2}$ and lemma
\ref{lem:left_right_ineq}, we have
\begin{equation}
\begin{split}
    \theta_{\mathbf{down}}&=\Delta\theta_{SE}(c_{\mathbf{down}})=(\Delta\theta_{SN}+\Delta\theta_{NE})(c_{\mathbf{down}})\\
    &\leq(\Delta\theta_{SN}+\Delta\theta_{MS})(c_{\mathbf{down}})=\Delta\theta_{MN}(c_{\mathbf{down}})<\pi,
\end{split}
\end{equation}
where the last inequality is given by proposition
\ref{thm:half_upper_bound}. Combine $\theta_{\mathbf{in}}<\pi$ from
proposition \ref{inout}. This contradicts
$\theta_{\mathbf{down}}+\theta_{\mathbf{in}}=2\pi$.

When $R_{\mathbf{start}}=R_{\mathbf{end}}\geq1$, use the equation
(\ref{eq:int_R_for_theta}) and the monotonicity with respect to $c$
for fixed range of $R$, from $c_{\mathbf{down}}<c_{\mathbf{in}}$, we
have
$\Delta\theta_{SM}(c_{\mathbf{in}})=\Delta\theta_{NE}(c_{\mathbf{in}})\leq\Delta\theta_{EM}(c_{\mathbf{down}})=\Delta\theta_{NS}(c_{\mathbf{down}})$.
Note that if $R_{\mathbf{start}}=R_{\mathbf{end}}=1$,
$\Delta\theta_{SM}(c_{\mathbf{in}})=\Delta\theta_{NE}(c_{\mathbf{in}})=\Delta\theta_{EM}(c_{\mathbf{down}})=\Delta\theta_{NS}(c_{\mathbf{down}})=0$.
Therefore,
\begin{equation}
\begin{split}
\theta_{\mathbf{in}}+\theta_{\mathbf{down}}&=(\Delta\theta_{SM}+\Delta\theta_{MN}+\Delta\theta_{NE})(c_{\mathbf{in}})+\Delta\theta_{SE}(c_\mathbf{down})\\
&\leq\Delta\theta_{MN}(c_{\mathbf{in}})+(\Delta\theta_{NS}+\Delta\theta_{SE}+\Delta\theta_{EM})(c_{\mathbf{down}})\\
&\leq\Delta\theta_{MN}(c_{\mathbf{in}})+\frac{1}{2}T(c_{\mathbf{down}})\leq \pi+\frac{\pi}{\sqrt2}<2\pi.
\end{split}
\end{equation}
Note that we use $\Delta\theta_{MN}(c)>\Delta\theta_{NM}(c)$ from
lemma \ref{lem:left_right_ineq} and
$\Delta\theta_{MN}(c)+\Delta\theta_{NM}(c)=T(c)$.
\end{proof}

If the bottom cell is a 3-cell, a 4-cell or a 5-cell,
$c_{\mathbf{down}}\geq c_*$. We can describe the change of angle in
terms of $h_1$, $h_2$, and $h_3$. Here we list all the possible
cases. The arrow indicates a triple junction with a ray. The point
on the phase plane will either jump from $D$ to $C$ or jump from $A$
to $B$.
\begin{center}
\begin{tabular}{|c|l|l|l|l|}
    \hline
    Cell & Path on the trajectory & $R_{\mathbf{start}}=R_{\mathbf{end}}<1$ & $R_{\mathbf{start}}<1<R_{\mathbf{end}}$ & $R_{\mathbf{start}}=R_{\mathbf{end}}>1$\\
  \hline
  5-cell & SD$\rightarrow$CD$\rightarrow$CD$\rightarrow$CE & $2h^\circ_1+2h_1$ & $h^\circ_1+3h_1+h_2+h^\circ_3$ & $4h_1+2h_2+2h^\circ_3$\\ \cline{2-5}
   & SA$\rightarrow$BD$\rightarrow$CD$\rightarrow$CE & & & \\
   & SD$\rightarrow$CA$\rightarrow$BD$\rightarrow$CE & & & \\
   & SD$\rightarrow$CD$\rightarrow$CA$\rightarrow$BE &  $2h^\circ_1+2h_1+2h_2$ & $h^\circ_1+3h_1+3h_2+h^\circ_3$ & $4h_1+4h_2+2h^\circ_3$\\ \cline{2-5}
   & SA$\rightarrow$BA$\rightarrow$BD$\rightarrow$CE & & &\\
   & SA$\rightarrow$BD$\rightarrow$CA$\rightarrow$BE & & &\\
   & SD$\rightarrow$CA$\rightarrow$BA$\rightarrow$BE & $2h^\circ_1+2h_1+4h_2$ & $h^\circ_1+3h_1+5h_2+h^\circ_3$ & $4h_1+6h_2+2h^\circ_3$\\ \cline{2-5}
   & SA$\rightarrow$BA$\rightarrow$BA$\rightarrow$BE & $2h^\circ_1+2h_1+6h_2$ & $h^\circ_1+3h_1+7h_2+h^\circ_3$ & $4h_1+8h_2+2h^\circ_3$\\
  \hline
  4-cell & SD$\rightarrow$CD$\rightarrow$CE & $2h^\circ_1+h_1$ &  $h^\circ_1+2h_1+h_2+h^\circ_3$ & $3h_1+2h_2+2h^\circ_3$\\ \cline{2-5}
   & SA$\rightarrow$BD$\rightarrow$CE & & &\\
   & SD$\rightarrow$CA$\rightarrow$BE & $2h^\circ_1+h_1+2h_2$ &  $h^\circ_1+2h_1+3h_2+h^\circ_3$ & $3h_1+4h_2+2h^\circ_3$\\ \cline{2-5}
   & SA$\rightarrow$BA$\rightarrow$BE & $2h^\circ_1+h_1+4h_2$ &
   $h^\circ_1+2h_1+5h_2+h^\circ_3$ & $3h_1+6h_2+2h^\circ_3$\\
  \hline
  3-cell & SD$\rightarrow$CE & $2h^\circ_1$ & $h^\circ_1+h_1+h_2+h^\circ_3$ & $2h_1+2h_2+2h^\circ_3$\\ \cline{2-5}
   & SA$\rightarrow$BE & $2h_1^\circ+2h_2$ & $h^\circ_1+h_1+3h_2+h^\circ_3$ & $2h_1+4h_2+2h^\circ_3$\\
  \hline
\end{tabular}
\end{center}

\begin{figure}[ht]
\includegraphics[height=4.5cm]{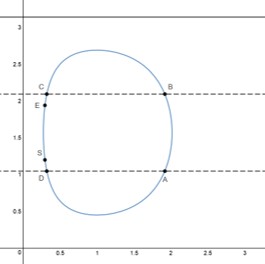}
\includegraphics[height=4.5cm]{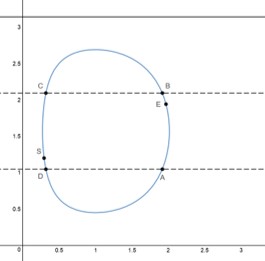}
\includegraphics[height=4.5cm]{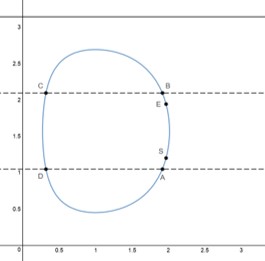}
\caption{Possible places for S and E}
\end{figure}

Since $\frac{\pi}{3}\leq\psi_{\mathbf{down}}<\frac{\pi}{2}$, $S$,
$E$ lie on either $CD$ or $AB$ arc of the trajectory corresponds to
$c_{\mathbf{down}}$. Use
$h_1^\circ=\Delta\theta_{SD}=\Delta\theta_{CE}$ when
$S(c_{\mathbf{down}})$ or $E(c_{\mathbf{down}})$ lie on the $CD$
arc. $h_3^\circ=\Delta\theta_{SB}=\Delta\theta_{AE}$ when
$S(c_{\mathbf{down}})$ or $E(c_{\mathbf{down}})$ lie on the $AB$
arc. Note that we can eliminate the case
$R_{\mathbf{start}}=R_{\mathbf{end}}=1$, we have
$c_{\mathbf{down}}=c_*$, $A=D=S$ and $B=C=E$. since when we goes
from the starting point to the triple junction, we either form a
complete loop or the curve will be degenerate.

\begin{thm}\label{thm:outout}
For the case the bottom cell is either a 3-cell, 4-cell or 5-cell,
it is impossible than $R_{\mathbf{start}}=R_{\mathbf{end}}>1$.
\end{thm}
\begin{proof}
In this case,
\begin{equation}
\begin{split}
    2h_3^\circ(c_{\mathbf{down}})&=(\Delta\theta_{SB}+\Delta\theta_{AE})(c_{\mathbf{down}})>\Delta\theta_{AB}(c_{\mathbf{down}})=h_3(c_{\mathbf{down}}).\\
\end{split}
\end{equation}
We have
\begin{equation}
\theta_{\mathbf{down}}\geq(2h_1+2h_2+2h^\circ_3)(c_{\mathbf{down}})>(2h_1+2h_2+h_3)(c_{\mathbf{down}})=(h_1+T)(c_{\mathbf{down}})>\frac{4\pi}{3}.
\end{equation}
This is impossible since $\theta_{\mathbf{up}}>\frac{2\pi}{3}$ and
$\theta_{\mathbf{down}}+\theta_{\mathbf{up}}=2\pi$.
\end{proof}

Therefore, for a regular shrinker, we have either
$R_{\mathbf{start}}<1$ or $R_{\mathbf{end}}<1$.

\begin{prop}\label{prop:range_for_gamma}
If $R_{\mathbf{start}}<1$ or $R_{\mathbf{end}}<1$, for a
$\Theta$-shaped regular shrinker, we have
$\theta_{\mathbf{up}}>(h_1+2h_2+2\Delta
\theta_{NA})(c_{\mathbf{up}})$. Moreover, for $c\geq\bar c=
e^{\frac{1.3065-1}{2}}$, $(h_1+2h_2+2\Delta \theta_{NA})(c)>\pi$.
Therefore, we have $c_{\mathbf{up}}\in I_A=(c_*,\bar c)$.  In this
case, we have $\theta_{\mathbf{up}}=\theta_{\mathbf{in}}\in
(0.9947\pi, \pi]$.
\end{prop}
\begin{proof}
Without loss of generality, assume $R_{\mathbf{start}}<1$. Recall
that
\begin{equation}
\theta_{\mathbf{up}}
=(\Delta\theta_{SA}+\Delta\theta_{BA}+\Delta\theta_{BE})(c_{\mathbf{up}})
=(h_1+2h_2)(c_{\mathbf{up}})+\Delta\theta_{SA}(c_{\mathbf{up}})+\Delta\theta_{BE}(c_{\mathbf{up}}).
\end{equation}
If $R_{\mathbf{end}}<1$, we have
\begin{equation}
\Delta\theta_{SA}(c_{\mathbf{up}})+\Delta\theta_{BE}(c_{\mathbf{up}})\geq\Delta\theta_{NA}(c_{\mathbf{up}})+\Delta\theta_{BM}(c_{\mathbf{up}})=2\Delta\theta_{NA}(c_{\mathbf{up}}).
\end{equation}
If $R_{\mathbf{end}}>1$, we have
\begin{equation}
\Delta\theta_{SA}(c_{\mathbf{up}})+\Delta\theta_{BE}(c_{\mathbf{up}})=\Delta\theta_{SN}(c_{\mathbf{up}})+\Delta\theta_{NA}(c_{\mathbf{up}})+\Delta\theta_{BM}(c_{\mathbf{up}})
-\Delta\theta_{EM}(c_{\mathbf{up}}).
\end{equation}
We want to compare $\Delta\theta_{SN}$ and $\Delta\theta_{EM}$.
Using the symmetry of the trajectory with respect to
$\psi=\frac{\pi}{2}$ and lemma \ref{lem:left_right_ineq},
$\Delta\theta_{SN}(c_{\mathbf{up}})\geq\Delta\theta_{EM}(c_{\mathbf{up}})$.
Therefore, we also have
$\Delta\theta_{SA}(c_{\mathbf{up}})+\Delta\theta_{BE}(c_{\mathbf{up}})\geq
2\theta_{NA}(c_{\mathbf{up}})$ and $\theta_{\mathbf{up}}\geq
h_1(c_{\mathbf{up}})+2h_2(c_{\mathbf{up}})+2\theta_{NA}(c_{\mathbf{up}})$.
From proposition \ref{inout}, we have
$(h_1+2h_2+2\theta_{NA})(c_{\mathbf{up}})\leq\theta_{\mathbf{up}}=\theta_{\mathbf{in}}\leq\pi$.


For $c\geq e^{0.19}$, we have $\eta\geq 1.38$, by using theorem
\ref{h12h2},  $h_1+2h_2+2\theta_{NA}\geq (h_1+2h_2)>\pi$.


For  $e^{\frac{1}{6}}\leq c<e^{0.19} $, we have
$\frac{4}{3}\leq\eta< 1.38$ and $k_2\geq 1$.  Recall that
\begin{equation}\label{ineq:h2+}
\begin{split}
\Delta\theta_{NA}(c)=\Delta\phi_{NA}(c)+\Delta\psi_{NA}(c)
=\int_{k_1}^{k_2}\frac{dk}{\sqrt{\eta-V(k)}}+\left(\sin^{-1}(\frac{1}{c})-\frac{\pi}{3}\right) ,
\end{split}
\end{equation}
where $V(k)=k^2-2\log k$, $k_1=\frac{1}{c}$,
$k_2=\frac{\sqrt{3}}{2}R^+(\frac{\sqrt{3}}{2}c)$ is the curvature at
$A$, and $\eta=1+2\log c$. Using  lemma \ref{lower2nd} with
$k_0=\frac{1}{c}$ and $H=V(\frac{1}{c})-(1+c)(\frac{1}{c}-1)^2$,  we
have $\frac{\eta-H}{1+k_0}=\frac{c-1}{c}$ and
\begin{equation}
\begin{split}
\Delta\theta_{NA}(c)
\geq \int_{k_1}^{1}\frac{dk}{\sqrt{\eta-V(k)}}+\left(\sin^{-1}(\frac{1}{c})-\frac{\pi}{3}\right)
\geq\frac{1}{\sqrt{1+c}}\sin^{-1}(\sqrt{\frac{c-1}{c}}))
+\left(\sin^{-1}\left(\frac{1}{c}\right)-\frac{\pi}{3}\right). \notag
\end{split}
\end{equation}
When  $c$ increases, $\frac{c-1}{c}$ increases. Since $k_2\geq 1$
and $e^{\frac{1}{6}}\leq c<e^{0.19} $, using theorem \ref{h12h2}, we
have
\begin{equation}
\begin{split}
2\Delta\theta_{NA}(c)\geq&\frac{2}{\sqrt{1+e^{0.19}}}\sin^{-1}\sqrt{\frac{e^{\frac{1}{6}}-1}{e^{\frac{1}{6}}}}
+2\left(\sin^{-1}(\frac{1}{e^{0.19}})-\frac{\pi}{3}\right)\geq0.1252\pi, \notag
\end{split}
\end{equation}
and $    (h_1+2h_2+2\Delta\theta_{NA})>0.9456\pi+0.1252\pi>\pi$ for
$e^{\frac{1}{6}}\leq c<e^{0.19} $.



For $ e^{\frac{1.31-1}{2}}\leq c< e^{\frac{1}{6}}$,  we have
$0.6235\leq k_{\mathbf{min}}=R^-(c)\leq 0.6358$ and $0.9590\leq
k_2=\frac{\sqrt{3}}{2}R^+(\frac{\sqrt{3}}{2} c)\leq 1$. Using Lemma
\ref{lower2nd} and the inequality (\ref{ineq:h2+}),
\begin{align}
(h_1+2h_2)(c)>&\frac{2}{\sqrt{1+\frac{1}{k_{\mathbf{min}}}}} \left(\frac{\pi}{2}-\sin^{-1}(\frac{1-k_2}{1-k_{\mathbf{min}}})\right)+\frac{\pi}{3} \notag  \\
\geq &\frac{2}{\sqrt{1+\frac{1}{0.6235}}} \left(\frac{\pi}{2}-\sin^{-1}(\frac{1-0.9590}{1-0.6358})\right)+\frac{\pi}{3}>0.9084\pi,  \notag
\end{align}
and
\begin{equation}\notag
\begin{split}
2\Delta\theta_{NA}(c)\geq&\frac{2}{\sqrt{1+e^{\frac{1}{6}}}}\Big[\sin^{-1}\sqrt{\frac{e^{\frac{1.31-1}{2}}-1}{e^{\frac{1.31-1}{2}}}}-\sin^{-1}\left(\sqrt{\frac{e^{\frac{1.31-1}{2}}}{e^{\frac{1.31-1}{2}}-1}}(1-0.9590)\right)\Big]\\
&+2\left(\sin^{-1}(\frac{1}{e^{\frac{1}{6}}})-\frac{\pi}{3}\right)\geq0.0966\pi.
\end{split}
\end{equation}
Therefore, $
(h_1+2h_2+2\Delta\theta_{NA})(c)>0.9084\pi+0.0966\pi>\pi $ for $
e^{\frac{1.31-1}{2}}\leq c< e^{\frac{1}{6}}$.

For $ e^{\frac{1.3065-1}{2}}=\bar c\leq c< e^{\frac{1.31-1}{2}}$, we
have $0.6356\leq k_{\mathbf{min}}=R^-(c)\leq 0.6377$ and $0.9513\leq
k_2=\frac{\sqrt{3}}{2}R^+(\frac{\sqrt{3}}{2} c)\leq 0.9591$. Using
lemma \ref{lower2nd} and the inequality (\ref{ineq:h2+}), we have
\begin{align}
(h_1+2h_2)(c)>&\frac{2}{\sqrt{1+\frac{1}{k_{\mathbf{min}}}}} \left(\frac{\pi}{2}-\sin^{-1}(\frac{1-k_2}{1-k_{\mathbf{min}}})\right)+\frac{\pi}{3} \notag  \\
\geq &\frac{2}{\sqrt{1+\frac{1}{0.6356}}} \left(\frac{\pi}{2}-\sin^{-1}(\frac{1-0.9513}{1-0.6377})\right)+\frac{\pi}{3}>0.9031\pi,  \notag
\end{align}
and
\begin{equation}
\begin{split}
2\Delta\theta_{NA}(c)\geq&\frac{2}{\sqrt{1+e^{\frac{1.31-1}{2}}}}\Big[\sin^{-1}\sqrt{\frac{\bar c-1}{\bar c}} -\sin^{-1}\sqrt{\frac{\bar c}{\bar c-1}}(1-0.9513)\Big]\\
&+2\left(\sin^{-1}(\frac{1}{e^{\frac{1.31-1}{2}}})-\frac{\pi}{3}\right)\geq0.0988\pi.
\end{split}
\end{equation}Therefore, $ (h_1+2h_2+2\Delta\theta_{NA})(c)>0.9031\pi+0.0988\pi>\pi $ for
$ \bar c \leq c< e^{\frac{1.31-1}{2}}$.

Combining above estimates, we obtain an contradiction for
$c_{\mathbf{up}}\geq \bar c$. Therefore, $c_{\mathbf{up}}\in I_A$.
Since the network is regular, 
$\psi_{\mathbf{down}}=\psi_{\mathbf{up}}+\frac{\pi}{3}$. Using the
conservation law (\ref{eq:conservation_law}),
\begin{equation}
\frac{\sin(\psi_{\mathbf{up}}+\frac{\pi}{3})}{\sin\psi_{\mathbf{up}}}=
\frac{\sin\psi_{\mathbf{down}}}{\sin\psi_{\mathbf{up}}}
=\frac{c_{\mathbf{up}}}{c_{\mathbf{down}}}
\leq \frac{\bar c}{c_*}.
\end{equation}
Since $\frac{\sin (x+\frac{\pi}{3})}{\sin x}$ decreases on the
interval $\in (0,\frac{\pi}{3}]$, we have
$\psi_{\mathbf{up}}>0.3307\pi$. Note that on $\gamma_{\mathbf{in}}$,
we have $\Delta\phi>0$. Hence,
\begin{equation}
\theta_{\mathbf{up}}=\theta_{\mathbf{in}}=\Delta\phi-\Delta\psi
\geq 0+(\psi_{\mathbf{in}}-(\pi-\psi_{\mathbf{in}}))
=\frac{\pi}{3}+2\psi_{\mathbf{up}}>0.9947\pi.
\end{equation}
Combining Proposition \ref{inout},
$\theta_{\mathbf{up}}=\theta_{\mathbf{in}}\in (0.9947\pi,\pi]$.
\end{proof}

\begin{thm}\label{thm:inout}
For the case that the bottom cell is either a 3-cell, a 4-cell, or a
5-cell, it is impossible either $R_{\mathbf{start}}>1$ or
$R_{\mathbf{end}}>1$.
\end{thm}
\begin{proof}
Assume the contrary, without loss of generality, let
$R_{\mathbf{start}}<1<R_{\mathbf{end}}$. For the case that the
bottom cell is either a 3-cell, 4-cell, or 5cell, from lemma
\ref{lem:left_right_ineq}, we have
$\Delta\theta_{SD}\geq\Delta\theta_{EB}$ and
\begin{equation}
\begin{split}
    (h_1^\circ+h_3^\circ)(c_{\mathbf{down}})&=(\Delta\theta_{SC}+\Delta\theta_{AE})(c_{\mathbf{down}})>(\Delta\theta_{EB}+\Delta\theta_{AE})(c_{\mathbf{down}})=h_3(c_{\mathbf{down}}).\\
\end{split}
\end{equation}
Therefore,
\begin{equation}
\begin{split}
\theta_{\mathbf{down}}\geq&(h_1^\circ+h_1+h_2+h_3^\circ)(c_{\mathbf{down}})\geq h_1(c_{\mathbf{down}})+h_3(c_{\mathbf{down}})
\geq h_1(\bar c)+h_3(\bar c),
\end{split}
\end{equation}
where the last inequality comes from that $h_1(c)$ and $h_3(c)$
decrease as $c$ increases.  Since $K(R^+(\bar c\sin\psi))=\bar
c\sin\psi\leq \bar c<K(\sqrt{2})=\frac{\sqrt{e}}{\sqrt{2}}$, we have
$\sqrt{2}>R^+( \bar c\sin\psi)$ and
\begin{equation} \label{ineq:h3barc}
h_3(\bar c)=\int^{\frac{2\pi}{3}}_{\frac{\pi}{3}}\frac{d\psi}{[R^+(\bar c\sin\psi)]^2-1}\geq\frac{\pi}{3}\frac{1}{(\sqrt{2})^2-1}=\frac{\pi}{3}.
\end{equation}
On the other hand,  using lemma \ref{lower2nd},
\begin{equation} \label{ineq:h1barc}
h_1(\bar c)=\int_{k_\mathbf{min}}^{k_2}\frac{2dk}{\sqrt{\bar\eta-V(k)}}+\frac{\pi}{3}
\geq \frac{2}{\sqrt{1+\frac{1}{k_{\mathbf{min}}}}}\left(
\frac{\pi}{2}-\sin^{-1}\left(\frac{1-k_2}{1-k_{\mathbf{min}}}\right)
\right)+\frac{\pi}{3},
\end{equation}
where $k_{\mathbf{min}}$  is the global minimum for curvature of
$\gamma_{\bar c}$, and $k_2$ is the curvature of $\gamma_{\bar c}$
at the point $D(\bar c)$, and $\bar\eta=1.3065$. By calculating,
$k_{\mathbf{min}}\in (0.6376,0.6377)$ and $k_2\in (0.7834,0.7835)$,
and $h_1(\bar c)\geq 0.7027\pi$. Therefore,
\begin{equation}
\theta_{\mathbf{in}}+\theta_{\mathbf{down}}\geq 0.9947\pi+(0.7027\pi+\frac{\pi}{3})>2\pi.
\end{equation}
It is impossible.
\end{proof}

From now on, for the case that the bottom cell is either a 3-cell, a
4-cell or a 5-cell, we have $R_{\mathbf{start}}=R_{\mathbf{end}}<1$.

\begin{thm}
There does not exist solution with bottom cell being a 5-cell.
\end{thm}
\begin{proof}
The smallest possible angle for $\theta_{\mathbf{down}}$ is that the
triple junctions are all of the $D\to C$ type. Since
$R_{\mathbf{start}}=R_{\mathbf{end}}<1$,
\begin{equation}
\theta_{\mathbf{down}}\geq(2h_1^\circ+2h_1)(c_{\mathbf{down}})\geq2h_1(c_{\mathbf{up}}).
\end{equation}

From theorem \ref{h12h2}, for every $c\geq c_*$, we have
$h_1(c)+2h_2(c)>0.7789\pi(>\frac{2\pi}{3})$. This equation is a
lower bound of $\theta_{\mathbf{up}}$.  From the proof of theorem
\ref{thm:inout}, we have $h_1(\bar{c})>0.7027\pi$. Therefore,
\begin{equation}
\theta_{\mathbf{up}}+\theta_{\mathbf{down}}>(h_1+2h_2)(c_{\mathbf{up}})+2h_1(c_{\mathbf{up}})>0.7789\pi+2\times 0.7027\pi>2\pi
\end{equation}
and we get a contradiction.
\end{proof}

\section{2 4-cells or a 4-cell and a 3-cell}
First, we consider the case which $\gamma_{\mathbf{down}}$ is $SA\to
BA\to BE$. In this case, there is a symmetry between
$\gamma_{\mathbf{down}}$ and $\gamma_{\mathbf{up}}$. Precisely
speaking, for any $R_0$, on the trajectory of energy c, define
$P(R_0)=(R_0,\sin^{-1}\frac{K(R_0)}{c})$. We have
$S=P(R_{\mathbf{start}})$ for $\gamma_{\mathbf{up}}$ and
$\gamma_{\mathbf{down}}$. The change of angle can be expressed as
\begin{equation}
\begin{split}
\theta_{\mathbf{up}}=&(h_1+2h_2+2\Delta \theta_{P(R_{\mathbf{start}})A})(c_{\mathbf{up}}),\\
\theta_{\mathbf{down}}=&(h_1+4h_2+2h^\circ_1)(c_{\mathbf{down}})
=
(h_1+2h_2+2\Delta\theta_{P(R_{\mathbf{start}})A})(c_{\mathbf{down}}).
\end{split}
\end{equation}
Note that
\begin{equation}
 \Delta \theta_{P(R_{\mathbf{start}})A}(c)=\int_{R_{\mathbf{start}}}^{R^+(\frac{\sqrt{3}}{2}c)} \frac{K(R)dR}{R\sqrt{c^2-K(R)^2}}.
\end{equation}
To obtain uniqueness and existence of the regular shrinker in this
case, we need the following lemmas.

\begin{lem}\label{lem:increase1}
Given a number $0.7\leq R_0\leq1$, $\Delta \theta_{P(R_0)A}(c)$
strictly increases on the admissible interval $I_A=[c_*,\bar c]$.
\end{lem}
\begin{proof}
As $c\in I_A$, using
$K(R^+(\frac{\sqrt{3}c}{2}))=\frac{\sqrt{3}}{2}c$, we have
$\frac{dR^+}{dc}=\frac{1}{c(R^+-\frac{1}{R^+})}$ and
\begin{equation}
\begin{split}
\frac{d\Delta\theta_{P(R_0)A}(c)}{dc}=&\frac{\frac{\sqrt{3}c}{2}}{R^+\sqrt{c^2-(\frac{\sqrt{3}c}{2})^2}}
\frac{1}{c(R^+-\frac{1}{R^+})}-\int_{R_0}^{R^+(\frac{\sqrt{3}c}{2})} \frac{cK(R)dR}{R\left(\sqrt{c^2-K^2(R)}\right)^3}.\\
\end{split}
\end{equation}
Let $W(R,c)=\frac{K(R)/c}{\left(\sqrt{1-(K(R)/c)^2}\right)^3}$. Note
that, fixed $R$, $W(R,c)$ is a decreasing function of $c$ and
$W(R,c_*)>0$ because of $K(R)<c_*$ for $R\in [0.7,1.1]$. Therefore,
\begin{equation}
\begin{split}
\frac{d\Delta \theta_{P(R_0)A}(c)}{dc}\geq& \frac{1}{c}\left(\frac{\sqrt{3}}{R^+(\frac{\sqrt{3}c}{2})^2-1}-\int_{R_0}^{R^+(\frac{\sqrt{3}c}{2})} \frac{W(R,c_*)dR}{R}\right)\\
\geq&\frac{1}{c}\left(\frac{\sqrt{3}}{1.1^2-1}-\int_{0.7}^{1.1} \frac{W(R,c_*)dR}{R}\right)\\
\geq& \frac{1}{c}\left(\frac{\sqrt{3}}{0.21}-\overset{4}{\underset{i=1}\Sigma } \left(\underset{R\in J_i}\max\frac{|J_i|}{R}\right) \left(\underset{R\in J_i}\max W(R,c_*)\right)
\right),
\end{split}
\end{equation}
where $J_1=[0.7,0.73]$, $J_2=[0.73,0.8]$, $J_3=[0.8,0.9]$,
$J_4=[0.9,1.1]$, and the second inequality comes from
$\underset{c\in I_A}\max R^+(\frac{\sqrt{3}c}{2})<1.1$ because of
$K(1.1)>\frac{\sqrt{3}}{2}\bar c$. Since $W(R,c_*)$ decreases for
$R\in [0.7,1]$ and increases for $R\in[1, 1.1]$, we have
\begin{equation}
\begin{split}
&\left(\underset{R\in J_1}\max\frac{|J_1|}{R}\right) \left(\underset{R\in J_1}\max W(R,c_*)\right)=\frac{0.03}{0.7}W(0.7,\frac{2}{\sqrt{3}})<1.81\\
&\left(\underset{R\in J_2}\max\frac{|J_2|}{R}\right) \left(\underset{R\in J_2}\max W(R,c_*)\right)=\frac{0.07}{0.73}W(0.73,\frac{2}{\sqrt{3}})<2.27\\
&\left(\underset{R\in J_3}\max\frac{|J_3|}{R}\right) \left(\underset{R\in J_3}\max W(R,c_*)\right)=\frac{0.1}{0.8}W(0.8,\frac{2}{\sqrt{3}})<1.47\\
&\left(\underset{R\in J_4}\max\frac{|J_4|}{R}\right) \left(\underset{R\in J_4}\max W(R,c_*)\right)=\frac{0.2}{0.9}\max \left\{W(0.9,\frac{2}{\sqrt{3}}), W(1.1,\frac{2}{\sqrt{3}})\right\}<1.74.
\end{split}
\end{equation}
Therefore,
\begin{equation}
\frac{d\Delta\theta_{P(R_0)A}(c)}{dc}\geq\frac{1}{c}\left(\frac{\sqrt{3}}{0.21}-(1.81+2.27+1.47+1.74)\right)>0.
\end{equation}
That is, $\Delta\theta_{P(R_0)A}(c)$ increases on the admissible
interval $I_A$.
\end{proof}

\begin{lem}\label{lem:increaseh12h2}
 $(h_1+2h_2)(c)$ and $h_2(c)$ are increasing on the admissible interval $I_A$.
\end{lem}
\begin{proof}
For any $c\in I_A$, let $Q^+$ and $Q^-$ be
$(0.7,\pi-\sin^{-1}\frac{K(0.7)}{c})$ and
$(0.7,\sin^{-1}\frac{K(0.7)}{c})$ on the $R-\psi$ plane.
\begin{equation}\label{eq:h12h2_6}
h_1(c)+2h_2(c)=\Delta\theta_{Q^+Q^-}(c)+2\Delta \theta_{Q^-A}(c)=f(c)+2\Delta \theta_{P(0.7)A}(c),
\end{equation}
where
\begin{equation}
f(c)=\int_{\sin^{-1}(\frac{K(0.7)}{c})}^{\pi-\sin^{-1}(\frac{K(0.7)}{c})} \frac{d\psi}{1-R^-(c\sin\psi)^2}.
\end{equation}
As $c\in I_A=[c_*,\bar c]$,
\begin{equation}
\begin{split}
f'(c)=&\frac{2}{c}\left(\frac{K(0.7)}{0.51}\frac{1}{\sqrt{c^2-K^2(0.7)}}-\int_{\sin^{-1}(\frac{K(0.7)}{c})}^{\pi-\sin^{-1}(\frac{K(0.7)}{c})} \frac{(R^-)^2}{(1-(R^-)^2)^3}d\psi \right)\\
\geq&\frac{2}{c}\left(\frac{K(0.7)}{0.51}\frac{1}{\sqrt{(\bar c)^2-K^2(0.7)}}-\frac{0.7^2}{(1-0.7^2)^3}\left(\pi-2\sin^{-1}\frac{K(0.7)}{c}\right)\right)>0,
\end{split}
\end{equation}
where the inequality holds since $\frac{2(R^-)^2}{(1-(R^-)^2)^3}$ is
a decreasing function of $\psi$. We have $f(c)$ increases strictly
on $I_A$. Combining the equation (\ref{eq:h12h2_6}) and using Lemma
\ref{lem:increase1}, $h_1(c)+2h_2(c)$ increases on $I_A$. Since
$h_1(c)$ is decreasing, the function $h_2(c)$ is increasing on
$I_A$.
\end{proof}

\begin{thm}\label{thm:existence_nondegenerate}
For the case that the bottom cell is a 4-cell with $(R,\psi)$ being
$SA\to BA\to BE$, there exists a unique solution. The curve
$\gamma_{\mathbf{in}}$ is a line segment through the origin and the
network is symmetric with respect to $\gamma_{\mathbf{in}}$.
\end{thm}
\begin{proof}
For this case, we have $R_{\mathbf{start}}=R_{\mathbf{end}}<1$,
therefore,
\begin{equation}
\theta_{\mathbf{down}}=(h_1+2h_2+2\Delta \theta_{P(R_{\mathbf{start}})A})(c_{\mathbf{down}}).
\end{equation}
Since $c_{\mathbf{up}}\in I_A$, we have
\begin{equation}
    R_{\mathbf{start}}>R^-(\frac{\sqrt{3}}{2}c_{\mathbf{up}})\geq R^-(\frac{\sqrt{3}}{2}\bar{c})>0.7.
\end{equation}

By lemma \ref{lem:increase1} and lemma \ref{lem:increaseh12h2},
since $c_*\leq c_{\mathbf{down}}\leq c_{\mathbf{up}}\leq\bar c$, we
have $\theta_{\mathbf{up}}\geq\theta_{\mathbf{down}}$ . On the other
hand, using proposition \ref{prop:range_for_gamma},
$\theta_{\mathbf{down}}=2\pi-\theta_{\mathbf{up}}\geq
2\pi-\pi=\pi\geq \theta_{\mathbf{up}}$. Therefore,
$\theta_{\mathbf{up}}=\theta_{\mathbf{down}}=\pi$,
$c_{\mathbf{down}}=c_{\mathbf{up}}$, and
$\psi_{\mathbf{up}}=\psi_{\mathbf{down}}=\frac{\pi}{3}$. Moreover,
$\theta_{\mathbf{down}}=(h_1+4h_2)(c_{\mathbf{down}})$.

 Use corollary \ref{cor:h1upper} and proposition \ref{prop:range_for_gamma},
\begin{equation}
    (h_1+4h_2)(c_*)=h_1(c_*)<\pi
\end{equation}
and
\begin{equation}
    (h_1+4h_2)(\bar c)>(h_1+2h_2)(\bar c)+2\Delta\theta_{NA}(\bar c)>\pi.
\end{equation}
By the continuity and the monotonicity of $h_1+4h_2$, there exists a
unique $c_0\in I_A$ such that
$c_{\mathbf{up}}=c_{\mathbf{down}}=c_0$ and
$\theta_{\mathbf{up}}=\theta_{\mathbf{down}}=h_1+4h_2={\pi}$.
\end{proof}

\begin{prop}
For the case that the bottom cell is a 3-cell or a 4-cell which is
not the previous case. There is no solution.
\end{prop}
\begin{proof}
For the case that the bottom cell $\gamma_{\mathbf{down}}$  is
either a 3-cell, or 4-cell which is not the special case, we have
\begin{equation}
\theta_{\mathbf{down}}<(h_1+2h_2+2\Delta \theta_{P(R_{\mathbf{start}})A})(c_{\mathbf{down}}).
\end{equation}

Again, we have $R_{\mathbf{start}}\in[0.7,1]$. Since $c_*\leq
c_{\mathbf{down}}\leq c_{\mathbf{up}} \leq\bar c$, we have
$\theta_{\mathbf{up}}\geq(h_1+2h_2+2\Delta
\theta_{P(R_{\mathbf{start}})A})(c_{\mathbf{down}})>\theta_{\mathbf{down}}$.
On the other hand, using proposition \ref{prop:range_for_gamma},
$\theta_{\mathbf{down}}=2\pi-\theta_{\mathbf{up}}\geq
2\pi-\pi=\pi\geq \theta_{\mathbf{up}}$. We obtain a contradiction.
\end{proof}

\section{Degenerate regular shrinkers}

We can find some degenerate regular shrinkers by allowing some edges
to be degenerate, which is a curve with zero length. The definition
of degernate regular shrinker can be found in \cite{MNPS}. The
theorems concerning the topology of the network in section 3 are
still applicable. Note that the curve with both ends attached to
rays cannot be degenerate, since the two rays can not form a
$\frac{\pi}{3}$ angle when they are not intersecting at the origin.
Therefore, the degenerate curves can only be the curves attached to
the starting or the ending point. Without loss of generality, assume
the first AL-curve on $\gamma_{\mathbf{up}}$ which goes out from the
starting point is degenerate. The angle
$\psi_{\mathbf{up}}=\frac{\pi}{3}$ and the starting point must be
either $D$ or $A$. For the other two curves, we have
$\psi_{\mathbf{in}}=\pi$ and $\gamma_{\mathbf{in}}$ must be a line
segment through the origin. Furthermore, we have
$\theta_{\mathbf{up}}=\theta_{\mathbf{down}}=\pi$ and
$c_{\mathbf{up}}=c_{\mathbf{down}}$. From proposition
\ref{prop:sym_start_end}, the ending point must be either the point
$B$ or $C$ on the trajectory.

The upper cell cannot be a 2-cell, otherwise, the starting point and
the ending point will be the same point and both
$\gamma_{\mathbf{up}}$ and $\gamma_{\mathbf{in}}$ will be
degenerate. It is impossible. Therefore, in the degenerate case, the
upper cell must be a 3-cell, a 4-cell or a 5-cell. Here, we use
$p$-cell to denote a cell with $p$ edges which are possibly
degenerate. From now on, we use the first curve, the second curve,
etc. to describe the smooth AL-curves when we traverse
$\gamma_{\mathbf{up}}$ from the starting point to the ending point.
Since $T(c)>\pi$ for any $c$, the curve can not have a complete
period on the trajectory. Note that if we find a solution for the
upper cell, since $\gamma_{\mathbf{in}}$ is a line segment on
$x$-axis, we can get a solution by letting $\gamma_{\mathbf{down}}$
be the reflection of $\gamma_{\mathbf{up}}$ with respect to
$x$-axis. Furthermore, if $R_{\mathbf{start}}=R_{\mathbf{end}}$, we
can get the solution by letting $\gamma_{\mathbf{up}}$ and
$\gamma_{\mathbf{down}}$ be symmetric with respect to the origin.

We need an estimation of angle to exclude some cases.
\begin{lem}\label{lem:2h1h2}
$(2h_1+h_2)(c)>\pi$ for any $c\geq c_*$.
\end{lem}
\begin{proof}
For the case $c\geq \hat c=e^{0.19}$, $\eta=1+2\log c\geq 1.38$.
Using theorem \ref{h12h2}, we obtain
\begin{align}
(2h_1+h_2)(c)> \frac{3}{2}h_1(c)+\frac{1}{2}(h_1+2h_2)(c)>\frac{3}{2}\times\frac{\pi}{3}+\frac{1}{2}\pi=\pi. \notag
\end{align}
On the other hand,  using lemma \ref{lower2nd},
\begin{equation}\label{ineq:h1hatc}
h_1( \hat c)=\int_{k_\mathbf{min}}^{k_2}\frac{2dk}{\sqrt{1.38-V(k)}}+\frac{\pi}{3}
\geq \frac{2}{\sqrt{1+\frac{1}{k_{\mathbf{min}}}}}\left(
\frac{\pi}{2}-\sin^{-1}\left(\frac{1-k_2}{1-k_{\mathbf{min}}}\right)
\right)+\frac{\pi}{3},
\end{equation}
where 
$k_2$ is the curvature of $\gamma_{\hat c}$ at the point $D(\hat
c)$. By using a scientific calculator, $k_{\mathbf{min}}\in
(0.6007,0.6008)$, $k_2\in (0.6871,0.6872)$, and $h_1( \hat c)\geq
0.5945\pi$. For the case  $c< \hat c$, we have $(2h_1+h_2)(c)\geq
2h_1(c)>2h_1(\hat c)>\pi$.
\end{proof}

\begin{thm} \label{thm:existence_3cell}
    If the upper cell is a 3-cell, the type of $\gamma_{\mathbf{up}}$ is $DD\to CB$ on the trajectory and $\theta_{\mathbf{up}}=h_1+h_2+h_3$.
\end{thm}
\begin{proof}
If the first curve is $DD$, which is degenerate, the second curve
starts from $C$. Since the second curve is neither degenerate nor
contain a complete period, the end point most be the point $B$. The
second curve is the $CB$ arc on the phase plane. In this case,
$\theta_{\mathbf{up}}=h_1+h_2+h_3$. Since
$\underset{c\to\infty}\lim(h_1+h_2+h_3)(c)=\frac{2\pi}{3}$ and
$(h_1+h_2+h_3)(c_*)=T(c_*)>\pi$, using intermediate value theorem,
there exist a solution.

If the first curve is $AA$, the second curve starts from $B$. If the
ending point is $C$,
$\theta_{\mathbf{up}}=h_2<\frac{T(c_{\mathbf{up}})}{2}<\pi$ is too
small. Otherwise, it will form a complete loop.
\end{proof}

\begin{figure}[h]
  \centering
    \includegraphics[width=3cm]{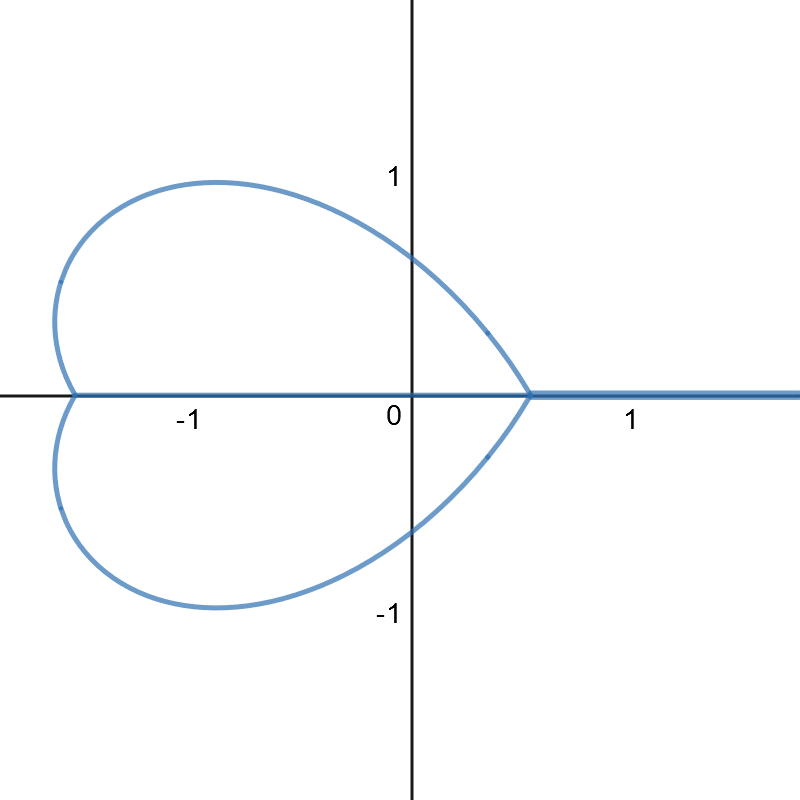}
  \caption{Degenerate case: heart. The ray on $x$-axis has multiplicity 2.}
\end{figure}

\begin{thm} \label{thm:4cell}
If the upper cell is a 4-cell
\begin{enumerate}
    \item If the first curve is $DD$, the type of $\gamma_{\mathbf{up}}$ should be $DD\to CA\to BC$ and $\theta_{\mathbf{up}}=h_1+2h_2$.
    \item If the first curve is $AA$, the types of $\gamma_{\mathbf{up}}$ are the following two cases:
    \begin{itemize}
        \item $AA\to BA\to BB$ and $\theta_{\mathbf{up}}=h_1+2h_2$.
        \item $AA\to BA\to BC$ and $\theta_{\mathbf{up}}=h_1+3h_2$. In this case, the energy belongs to $I_A$.
    \end{itemize}
\end{enumerate}
\end{thm}

\begin{proof}
(1) If the first curve is $DD$, which is degenerate, the second
curve starts at $C$.
\begin{itemize}
\item If the second curve ends at $D$, the third curve starts at $C$. If it immediately end at $C$, $\theta_{\mathbf{up}}=h_1<\pi$ from corollary \ref{cor:h1upper}. If the ending point is $B$, we have $\theta_{\mathbf{up}}=2h_1+h_2+h_3>2h_1+h_2>\pi$ from lemma \ref{lem:2h1h2}. Both cases are impossible.

\item If the second curve ends at $A$, the third curve starts at $B$. If it immediately ends at $B$, using lemma \ref{lem:left_right_ineq} and proposition \ref{thm:half_upper_bound}, $\theta_{\mathbf{up}}=h_1+h_2<\Delta\theta_{MN}<\pi$. If it ends at $C$, $\theta_{\mathbf{up}}=h_1+2h_2$. Since $2h_1+4h_2=2\pi$ has a unique solution which corresponds to the lens in the classification of regular shrinker with 1 closed region in \cite{CG}, we have existence and uniqueness for this case.
\end{itemize}
(2) If the first curve is $AA$,  the second curve starts at $B$.
\begin{itemize}
\item If the second curve ends at $D$, the third curve starts at $C$. If it immediately ends at $C$, $\theta_{\mathbf{up}}=h_1+h_2<\Delta\theta_{MN}<\pi$. This is too small. If the third curve ends at $B$, $\theta_{\mathbf{up}}=2h_1+2h_2+h_3$. This is is more than a period. Both cases are impossible.

\item If the second curve ends at $A$, the third curve starts from $B$. If it immediately ends at $B$, $\theta_{\mathbf{up}}=h_1+2h_2$. Again, from \cite{CG}, we have the existence and uniqueness. If the third curve ends at $C$, $\theta_{\mathbf{up}}=h_1+3h_2$. Since the upper curve is $SA\to BA\to BE$, proposition \ref{prop:range_for_gamma} is applicable and we have $c_{\mathbf{up}}\in I_A$. Using lemma \ref{lem:increaseh12h2}, we obtain $(h_1+3h_2)(c)$ is increasing on $I_A$.
Using corollary \ref{cor:h1upper} and proposition
\ref{prop:range_for_gamma},
\begin{equation}
\begin{split}
&(h_1+3h_2)(c_*)=h_1(c_*)<\pi,\\
&(h_1+3h_2)(\bar c)>(h_1+2h_2)(\bar c)+2\Delta \theta_{NA}(\bar c)>\pi.
\end{split}
\end{equation}
Therefore, there exists a unique number $c_{\mathbf{up}}$ such
that   $(h_1+3h_2)(c_{\mathbf{up}})=\pi$.
\end{itemize}
\end{proof}

\begin{figure}[h]
  \centering
    \includegraphics[width=3cm]{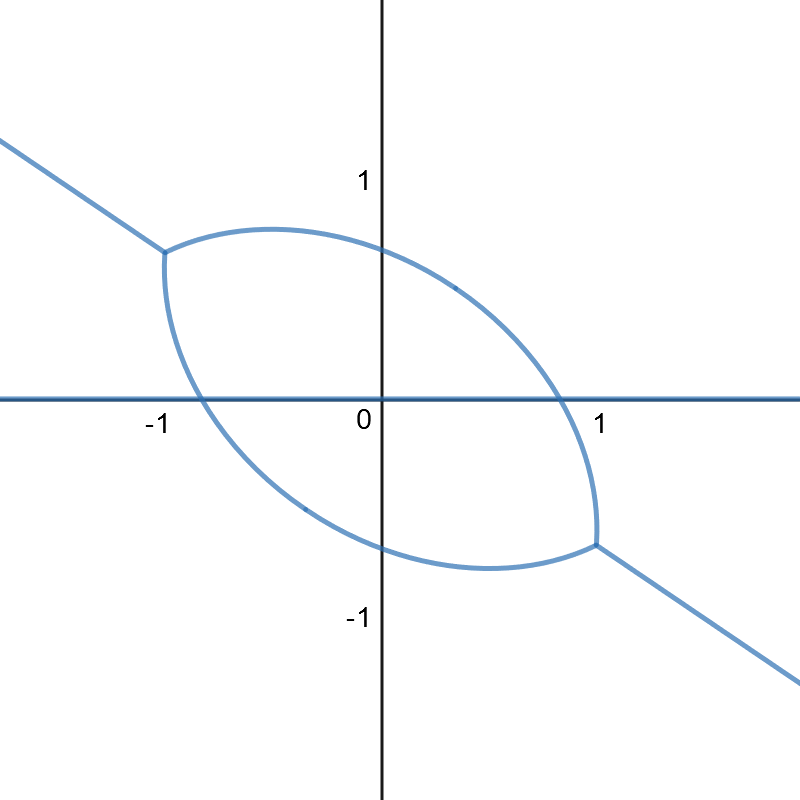}
    \includegraphics[width=3cm]{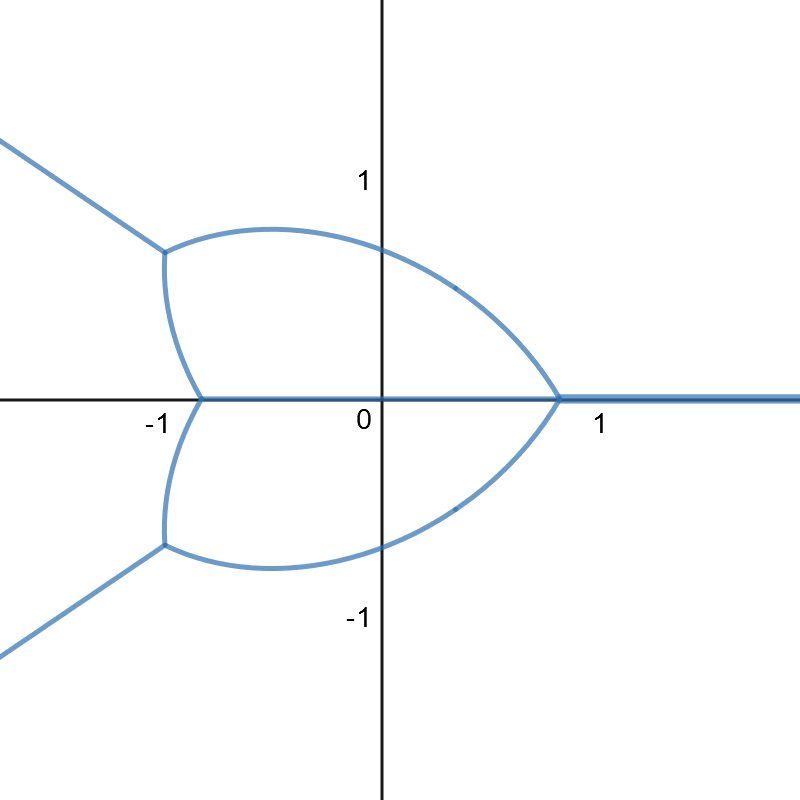}
    \includegraphics[width=3cm]{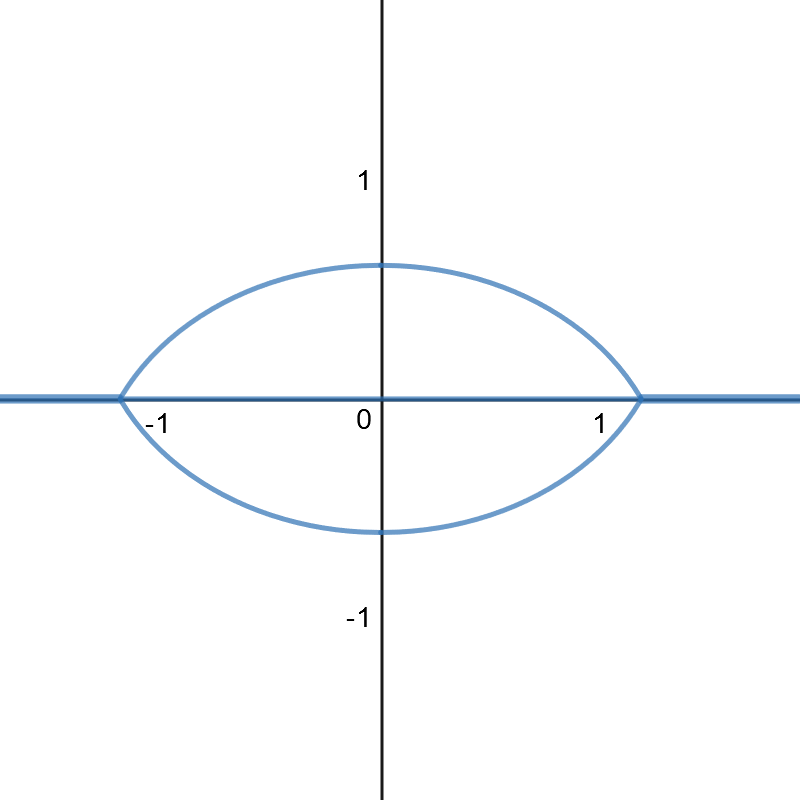}
    \includegraphics[width=3cm]{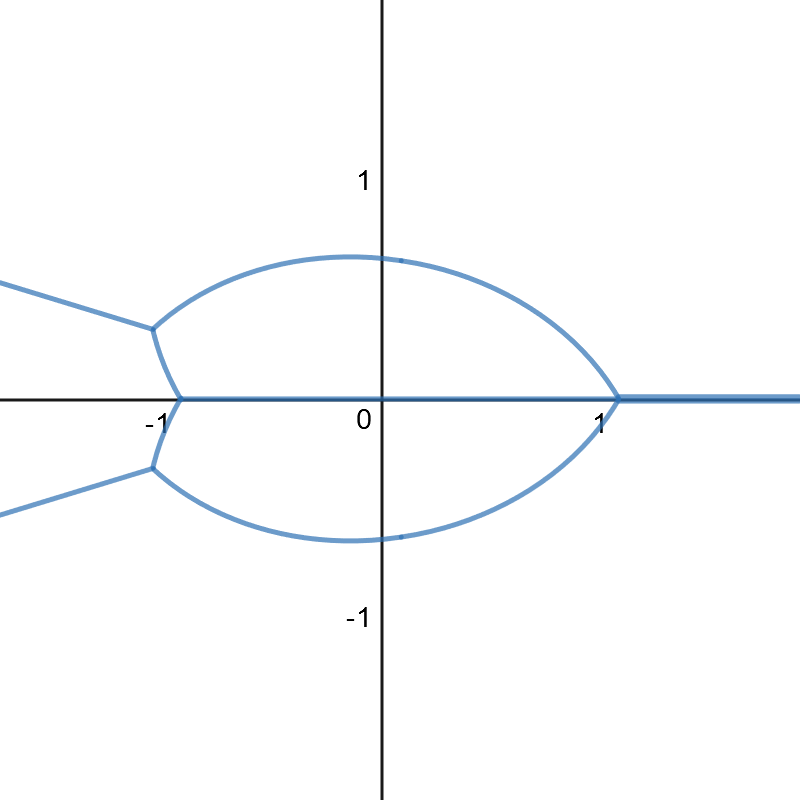}
  \caption{Degenerate cases: broken lens, cat, half lens, fox.}
\end{figure}
\begin{rmk}
    The broken lens is the only degenerate regular shinker with multiplicity 1.
\end{rmk}

\begin{thm} \label{thm:5cell}
If the upper cell is a 5-cell, the type of $\gamma_{\mathbf{up}}$ is
$DD\to CD\to CD\to CC$ and $\theta_{\mathbf{up}}=2h_1$.
\end{thm}

\begin{proof}
If the first curve is $DD$,  which is degenerate, the second curve
starts from $C$.
 \begin{itemize}
    \item If the second curve ends at $D$, the third curve starts at $C$ and ends at either $D$ or $A$. If it ends at $D$, suppose the fourth curve is not degenerate, $\theta_{\mathbf{up}}\geq 3h_1>\pi$. Therefore, the fourth curve is degenerate and $\theta_{\mathbf{up}}=2h_1$. Note that $h_1(c_*)>\frac{\pi}{2}$, $\underset{c\to\infty}{\lim} h_1(c)=\frac{\pi}{3}$ and $h_1$ is decreasing, we have existence and uniqueness in this case. If the third curve ends at $A$, $\theta_{\mathbf{up}}\geq2h_1+h_2>\pi$ from lemma \ref{lem:2h1h2}. This is impossible.

    \item If the second curve ends at $A$, the third curve must start from $B$ and end at either $D$ or $A$ and $\theta_{\mathbf{up}}\geq 2h_1+2h_2>\pi$. This is impossible.
\end{itemize}
If the first curve is $AA$,  the second curve starts from $B$.
\begin{itemize}
    \item If the second curve ends at $D$, the third curve starts at $C$ and ends at $D$. We have $\theta_{\mathbf{up}}\geq2h_1+h_2>\pi$  from lemma \ref{lem:2h1h2}. This is impossible.

    \item If the second curve ends at $A$, the third curve start at $B$ and ends at $D$ or $A$ and need to cross $CD$ arc. Therefore $\theta_{\mathbf{up}}\geq 2h_1+2h_2>\pi$ from lemma \ref{lem:2h1h2}, which is impossible.
\end{itemize}

\end{proof}
\begin{figure}[h]
  \centering
    \includegraphics[width=3cm]{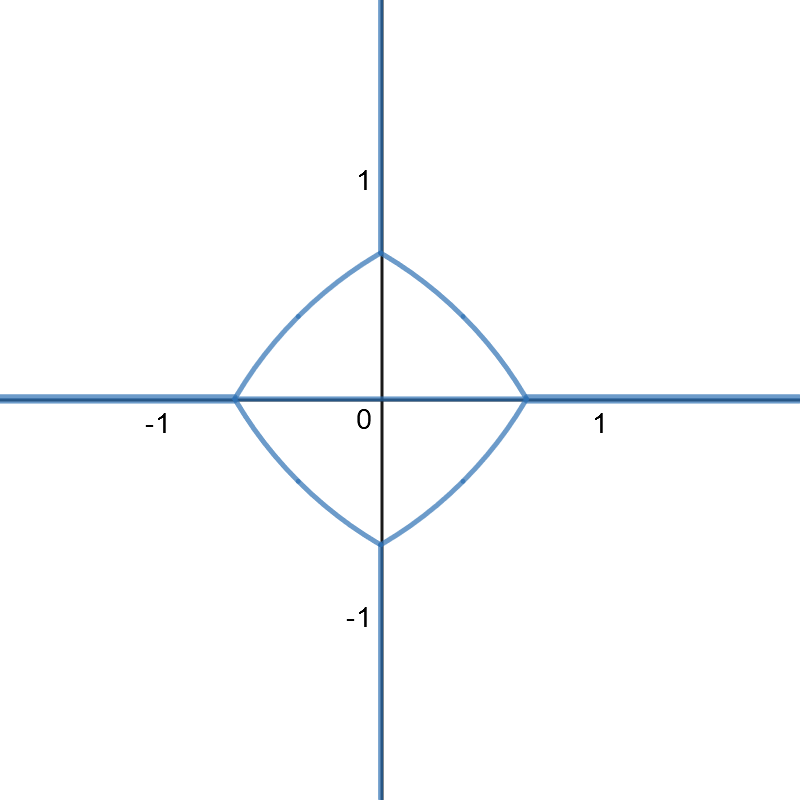}
  \caption{Degenerate case: half 4-ray star}
\end{figure}

\begin{lem} \label{lem:combine}
The degenerate regular shrinkers are either symmetric with respect
to $x$-axis or symmetric with respect to the origin.
\end{lem}

\begin{proof}
From theorem \ref{thm:existence_nondegenerate},
\ref{thm:existence_3cell}, \ref{thm:4cell}, \ref{thm:5cell}, the
possible curves $\gamma_{\mathbf{up}}$ in the (degenerate) regular
shrinker with 2 closed regions are the following six cases:
\begin{enumerate}
\item $DD\to CA\to BC $ with $h_1+2h_2=\pi$ and the energy $c_1$.
\item $DD\to CD\to CD\to DD$ with $2h_1=\pi$ and the energy $c_2$.
\item $DA\to BA\to BC$ with $h_1+4h_2=\pi$ and the energy $c_3$.
\item $DD\to CB$ with $h_1+h_2+h_3=\pi$ and the energy $c_4$.
\item $AA\to BA \to BC $ with $h_1+3h_2=\pi$ and the energy $c_5$.
\item $AA\to BA\to BB$ with  $h_1+2h_2=\pi$ and the energy $c_6$.
\end{enumerate}
In the cases (1), (2), and (3),
$R_{\mathbf{start}}=R_{\mathbf{end}}<1$. In the case (4) and (5), we
have either $R_{\mathbf{start}}<1<R_{\mathbf{end}}$ or
$R_{\mathbf{start}}>1>R_{\mathbf{end}}$. In the case (6), we have
$R_{\mathbf{start}}=R_{\mathbf{end}}>1$. If there exists a
degenerate regular shrinker with 2 closed regions, which
$\gamma_{\mathbf{up}}$ and $\gamma_{\mathbf{down}}$ are of different
types, it should be one of the following cases: $c_1=c_2$,
$c_1=c_3$, $c_2=c_3$, $c_4=c_5$.

If $c_1=c_2$ and $(h_1+2h_2)(c_1)=2h_1(c_2)=\pi$, from theorem
\ref{h12h2}, we have $c_1<e^{0.19}$. Since $h_1$ is decreasing, from
the inequality (\ref{ineq:h1hatc}),
$\frac{\pi}{2}=h_1(c_2)=h_1(c_1)>h_1(e^{0.19})>0.5945\pi$. This is a
contradiction.

If $c_1=c_3$ and $(h_1+2h_2)(c_1)=(h_1+4h_2)(c_3)=\pi$,
$h_2(c_1)=h_2(c_3)=0$, and $h_1(c_1)=\pi$. We obtain a contradiction
from corollary \ref{cor:h1upper}.

If $c_2=c_3$ and $2h_1(c_2)=(h_1+4h_2)(c_3)=\pi$, we obtain
$h_1(c_2)=\frac{\pi}{2}$ and $h_2(c_2)=\frac{\pi}{8}$. From theorem
\ref{h12h2}, it gives $0.7789\pi<(h_1+2h_2)(c_2)=0.75\pi$. This is a
contradiction.

If $c_4=c_5$ and $(h_1+h_2+h_3)(c_4)=(h_1+3h_2)(c_5)=\pi$, from
theorem \ref{thm:4cell}, $c_5\in I_A$. Using lemma
\ref{lem:increaseh12h2}, the function $h_2$ is increasing on $I_A$
and the function $(h_1+h_2+h_3)(c)=T(c)-h_2(c)$ is decreasing on
$I_A$ from that $T$ is decreasing. Therefore, using the inequalities
(\ref{ineq:h3barc}) and (\ref{ineq:h1barc}), we obtain
$\pi=(h_1+h_2+h_3)(c_5)\geq (h_1+h_2+h_3)(\bar
c)>0.7027\pi+0+\frac{\pi}{3}>\pi$. This is a contradiction.
\end{proof}
By combining theorems \ref{thm:existence_3cell}, \ref{thm:4cell},
\ref{thm:5cell}, and lemma \ref{lem:combine}, we find some
degenerate regular shrinkers with 2 closed regions, which are the
heart, the broken lens, the cat, the half lens, the fox, the half
4-ray star.

\end{document}